\documentclass[sts,reqno]{imsart}

\usepackage{amsmath, amssymb, color}
\usepackage{amsthm} 
\usepackage[colorlinks,citecolor=blue,urlcolor=blue]{hyperref}
\usepackage{xcolor,colortbl}
\usepackage{color}
\usepackage{amsfonts, fancybox}
\usepackage{amssymb}
\usepackage[american]{babel}
\usepackage{psfrag,color}
\usepackage{dcolumn}
\usepackage[format=hang, justification=justified, singlelinecheck=false]{subcaption}
\usepackage{flafter, bm, dsfont}
\usepackage[section]{placeins}

\usepackage{multirow}
\usepackage{color}
\usepackage{amsmath}
\usepackage{mathrsfs}
\usepackage{amsfonts}
\usepackage{dsfont}
\usepackage{amssymb}
\usepackage{graphicx}

\usepackage{cite}
\newcommand{\bc}{\begin{center}}
\newcommand{\ec}{\end{center}}
\newcommand{\ba}{\begin{array}}
\newcommand{\ea}{\end{array}}
\newcommand{\be}{\begin{eqnarray}}
\newcommand{\ee}{\end{eqnarray}}
\newcommand{\bel}{\begin{eqnarray}\label}
\newcommand{\eel}{\end{eqnarray}}
\newcommand{\bes}{\begin{eqnarray*}}
\newcommand{\ees}{\end{eqnarray*}}
\newcommand{\bn}{\begin{enumerate}}
\newcommand{\en}{\end{enumerate}}

\newcommand{\mse}{\mathsf{MSE}}

\definecolor{MIT}{cmyk}{.24, 1.00, .78, .17} 
\definecolor{pink}{cmyk}{0, 1, 0, 0} 
\definecolor{darkgreen}{cmyk}{1,0, 1, 0}

\newtheorem{theorem}{Theorem}
\newtheorem*{theorem*}{Theorem}
\newtheorem{lemma}[theorem]{Lemma}
\newtheorem{definition}[theorem]{Definition}
\newtheorem{proposition}[theorem]{Proposition}
\newtheorem*{proposition*}{Proposition}

\newtheorem{corollary}[theorem]{Corollary}

\newcommand{\diag}{\mathop{\mathsf{diag}}}

\newcommand{\MSE}{\mathop{\mathsf{MSE}}}

\usepackage[utf8]{inputenc}
\newcommand{\leadeq}[2][4]{\MoveEqLeft[#1] #2 \nonumber}

\newcommand{\diff}{\mathop{}\!\mathrm{d}}
\usepackage{stmaryrd}
\usepackage{bbm}
\usepackage[abbreviations]{foreign}
\newcommand\undermat[2]{%
  \makebox[-5pt][l]{$\smash{\underbrace{\phantom{%
    \begin{matrix}#2\end{matrix}}}_{\text{$#1$}}}$}#2}

\newcommand{\mockalph}[1]{}
\usepackage{mathtools}
\mathtoolsset{showonlyrefs}
\makeatletter
\newcommand{\subalign}[1]{%
  \vcenter{%
    \Let@ \restore@math@cr \default@tag
    \baselineskip\fontdimen10 \scriptfont\tw@
    \advance\baselineskip\fontdimen12 \scriptfont\tw@
    \lineskip\thr@@\fontdimen8 \scriptfont\thr@@
    \lineskiplimit\lineskip
    \ialign{\hfil$\m@th\scriptstyle##$&$\m@th\scriptstyle{}##$\crcr
      #1\crcr
    }%
  }
}
\makeatother

\newcommand{\cB}{\mathcal{B}}

\newcommand{\cF}{\mathcal{F}}

\newcommand{\cN}{\mathcal{N}}

\newcommand{\bone}{\mathbf{1}}

\newcommand{\R}{{\rm I}\kern-0.18em{\rm R}}
\newcommand{\h}{{\rm I}\kern-0.18em{\rm H}}
\newcommand{\K}{{\rm I}\kern-0.18em{\rm K}}
\newcommand{\p}{{\rm I}\kern-0.18em{\rm P}}
\newcommand{\E}{{\rm I}\kern-0.18em{\rm E}}
\newcommand{\Z}{{\rm Z}\kern-0.18em{\rm Z}}
\newcommand{\1}{{\rm 1}\kern-0.24em{\rm I}}
\newcommand{\N}{{\rm I}\kern-0.18em{\rm N}}

\newcommand{\pn}{\p_{\kern-0.25em n}}
\newcommand{\pnm}{\p_{\kern-0.25em n,m}}
\newcommand{\psubm}{\p_{\kern-0.25em m}}
\newcommand{\psubp}{\p_{\kern-0.25em p}}
\newcommand{\cfi}{\cF_{\kern-0.25em \infty}}
\newcommand{\ic}{\mathrm{i}}

\newcommand{\argmin}{\mathop{\mathrm{argmin}}}

\newcommand{\sign}{\mathop{\mathrm{sign}}}

\newcommand{\eps}{\varepsilon}

\newlength{\minipagewidth}
\begin{document}

% \maketitle

\begin{frontmatter}

\title{Optimal rates for total variation denoising}
\runtitle{TV denoising}

\author{ \fnms{Jan-Christian} \snm{Hütter}\thanksref{}\ead[label=huetter]{huetter@math.mit.edu} \and
 \fnms{Philippe} \snm{Rigollet}\thanksref{t2}\ead[label=rigollet]{rigollet@math.mit.edu}
}
\affiliation{Massachusetts Institute of Technology}

\thankstext{t2}{Supported in part by NSF grants DMS-1317308 and CAREER-DMS-1053987.}

\address{{Jan-Christian Hütter}\\
{Department of Mathematics} \\
{Massachusetts Institute of Technology}\\
{77 Massachusetts Avenue,}\\
{Cambridge, MA 02139-4307, USA}\\
\printead{huetter}
}

\address{{Philippe Rigollet}\\
{Department of Mathematics} \\
{Massachusetts Institute of Technology}\\
{77 Massachusetts Avenue,}\\
{Cambridge, MA 02139-4307, USA}\\
\printead{rigollet}
}

\runauthor{H\"utter, Rigollet}

\begin{abstract}
Motivated by its practical success, we show that the 2D total variation denoiser satisfies a sharp oracle inequality that leads to near optimal rates of estimation for a large class of image models such as bi-isotonic, H\"older smooth and cartoons.
Our analysis hinges on properties of the unnormalized Laplacian of the  two-dimensional grid such as eigenvector delocalization and spectral decay. We also present extensions to more than two dimensions as well as several other graphs.
\end{abstract}


\end{frontmatter}

\section{Introduction} % {{{2
\label{sec:introduction}

Total variation image denoising has known a spectacular practical success since its introduction by \cite{RudOshFat92} more than two decades ago. Surprisingly, little is known about its statistical performance. In this paper, we close this gap between theory and practice by providing a novel analysis for this estimator in a Gaussian white noise model. In this model, we observe a vector $y \in \R^n$ defined as
\begin{equation}
\label{eq:al}
y = \theta^\ast + \varepsilon\,,
\end{equation}
where $\theta^* \in \R^n$ is the unknown parameter of interest and $\eps \sim \cN(0, \sigma^2I_n)$ is a Gaussian random vector. In practice, $\theta^*$ corresponds to a vectorization of an image and we observe it corrupted by the noise $\eps$. The goal of image denoising is to estimate $\theta^*$ as accurately as possible. In this paper, we follow the standard employed in the image denoising literature and  measure the performance of an estimator $\hat \theta$ by its \emph{mean squared error}. It is defined by 
$$
\mse(\hat \theta):=\frac1n\|\hat \theta -\theta^*\|_2^2.
$$

Note that a lot of the work concerning the fused Lasso in the context of graphs has been focused on sparsistency results, \ie, conditions under which we can expect to recover the set of edges along which the signal has a jump \cite{HarLev12,QiaJia12,ShaSinRin12,OllVia15,ViaLamHoe16}, which is a different objective than controlling the MSE.

The total variation (TV) denoiser $\hat \theta$ is defined as follows. Let $G = (V, E)$ be an undirected connected graph with vertex set $V$ and edge set $E$ such that \(|V|=n, |E| = m \). The graph $G$ traditionally employed in image denoising is the two-dimensional (2D) grid graph defined as follows. The vertex set is $V=[N]^2$ and the edge set $E \subset [N]^2 \times [N]^2$ contains edge $e=\big([i,j], [k,l]\big)$ if and only if $[k,l]-[i,j] \in \{[1,0], [0,1]\}$. Nevertheless, our results remain valid for other graphs as discussed in Section~\ref{sec:other-graphs} and we work with a general graph $G$ unless otherwise mentioned.% (see Figure~\ref{FIG:grid}).

Throughout this paper it will be convenient to represent a graph $G$ by its edge-vertex incidence matrix $D=D(G) \in \{-1,0,1\}^{m \times n}$. Without loss of generality, identify $V$ to $[n]$ and $E$ to $[m]$ whenever convenient.   To each edge $e=(i,j) \in E$ corresponds a row $D_{e,:}$ of $D$ with entries given as follows. The $k$th entry $D_{e,k}$ of $D_{e,:}$ is given by
$$
D_{e,k}=\left\{
\begin{array}{rl}
1 &\text{if}\ k=\min(i, j)\\
-1 &\text{if}\ k=\max(i,j)\\
0 &\text{otherwise.}
\end{array}
\right.
$$
Note that the matrix $L=D^\top D$ is the \emph{unnormalized Laplacian} of the graph $G$~\cite{Chu97}. It can be represented as $L=\diag(A\1_n)-A$, where $A$ is the adjacency matrix of $G$ and $\diag(A\bone_n)$ is the diagonal matrix with $j$th diagonal element given by the degree of vertex $j$.

The TV denoiser $\hat \theta$ associated to $G$ is then given by any solution to the following minimization problem
\begin{equation}
	\label{eq:y}
	\hat\theta  \in \argmin_{\theta \in \R^n} \frac{1}{n} \| \theta - y \|_2^2 + \lambda \| D \theta \|_1\,,
\end{equation}
where $\lambda>0$ is a regularization parameter to be chosen carefully. Our results below give a precise choice for this parameter.
Note that \eqref{eq:y} is a convex problem that may be solved efficiently (see \cite{ArnTib16} and references therein).

Akin to the sparse case, the TV penalty in~\eqref{eq:y} is a convex relaxation for the number of times $\theta$ changes values along the edges of $G$. Intuitively, this is a good idea if $\theta^*$ takes small number of values for example. In this paper, we favor an analysis where $\theta^*$ is not of such form but may be well approximated by a piecewise constant vector. Our main result, Theorem~\ref{thm:Lasso-rates}, is a sharp oracle inequality that trades off approximation error against estimation error. In Section~\ref{sec:applications}, we present several examples where approximation error can be explicitly controlled: H\"older functions, Isotonic matrices and cartoon images. In each case, our results are near optimal in a minimax sense.

Our analysis partially leverages insight gained from recent results for the one-dimensional case where $G$ is the path graph by \cite{DalHebLed14}. In  this case, the TV denoiser is often referred to as a fused (or fusion) Lasso~\cite{TibSauRos05, Rin09}. Moreover, the TV denoiser $\hat \theta$ defined in~\eqref{eq:y} is often called to \emph{generalized fused Lasso}. The analysis provided in~\cite{DalHebLed14} is specific to the path graph and does not extend to more general graphs.  We extend these results to other graphs, with particular emphasis on the 2D grid. Critically, our analysis can be extended to graphs with specific spectral properties, such as random graphs with bounded degree. It is worth mentioning that our techniques, unfortunately do not recover the results of \cite{DalHebLed14} for the path graph.

\subsection{Notation} % {{{3
\label{sub:notation}

For two integers \( n, m \in \N \), we write \( [n] = \{1, \dots, n \} \),  \( \llbracket n, m \llbracket = \{n, n+1, \dots, m-1 \} \) and \( \llbracket n, m \rrbracket = \{n, n+1, \dots, m \} \). 
Moreover, for two real numbers $a,b$, we write $a\vee b=\max(a,b)$ and $a\wedge b=\min(a,b)$.

We  reserve bold-face letters like \( \bm{i}, \bm{j}, \bm{k} \) for multi-indices whose elements are written in regular font, \eg \( \bm{i} = (i_1, \dots, i_d) \).

We denote by $\bone_d$ the all-ones vector of $\R^d$. 

We write $\1(\cdot)$ for the indicator function.

For any two sets $A,B \subset \R^d$ we define their Minkowski sum as $A+B=\{a+b\,, a \in A, b \in B\}$. Moreover, for any $\eta \ge 0$, we denote by $\cB(\eta)=\{x \in \R^d\,, \|x\|\le \eta\}$ the Euclidean ball of radius $\eta$.

For any vector $x \in \R^d$, $T \subset [d]$, we define  $x_T \in \R^d$ to be the vector with $j$ coordinate given by $(x_T)_j=x_j \1(j \in T)$.

We denote by $A^\dagger$ the Moore-Penrose pseudo-inverse of a matrix $A$ and by \( \otimes \) the Kronecker product between matrices, \( (A \otimes B)_{p(r-1)+v,q(s-1)+w} = A_{r,s}B_{v,w} \).

The notation \( \lesssim \) means that the left-hand side is bounded by the right-hand side up to numerical constant that might change from line to line.
Similarly, the constants \( C, c \) are generic as well and are allowed to change.

\subsection{Previous work}
Despite an overwhelming practical success, theoretical results for the TV denoiser on the 2D grid have been very limited. \cite{mammen_locally_1997} obtained the first suboptimal statistical rates and more recent advances were made in~\cite{NeeWar13b} and~\cite{WanShaSmo15}.

First and foremost, both \cite{mammen_locally_1997} and \cite{WanShaSmo15} study the more general framework of \emph{trend filtering} where instead of applying the difference operator $D$ in the penalty, one may apply $D^{k+1}$ (with appropriate corrections due to the shrinking dimension of the image space). In this paper, we focus on the case where $k=0$. 

Second, while our paper focuses on fast rates (of the order $1/n$),~\cite{WanShaSmo15} also studies graphs that lead to slower rates. A prime example is the path graph that is omitted from the present work and for which~\cite{WanShaSmo15} recover the optimal rate $n^{-2/3}$ for signals $\theta^\ast$ such that $\|D\theta^\ast\|_1\le C$. This rate was previously known to be optimal~\cite{DonJoh95} for such signals, using comparison with Besov spaces. Remarkably, if $\theta^\ast$ is piecewise constant with large enough pieces,~\cite{DalHebLed14} proved that this rate can be improved to a rate of order $1/n$ using a rather delicate argument. Moreover, their result is also valid in a oracle sense, allowing for model misspecification and leading to adaptive estimation of smooth functions on the real line. Part of our results extend this application to higher dimensions.
%
%Third, it was also noticed by~\cite{WanShaSmo15} that fast rates arise as soon as the underlying graph satisfies two conditions: (i) bounded degree and (ii) constant spectral gap. We obtain similar results for these graphs but our techniques lead to slightly tighter bounds that are strictly better for graphs with large maximum degree. 

Our paper improves upon the work of~\cite{mammen_locally_1997} and~\cite{WanShaSmo15} in three directions. First,  our analysis leads to an optimal fast rate of order $\|D\theta^*\|_1/n$ for the 2D grid, unlike the rates \( (\| D \theta^\ast \|_1/n)^{3/5} \) and $(\|D\theta^*\|_1/n)^{4/5}$ that were obtained by~\cite{mammen_locally_1997} and~\cite{WanShaSmo15}, respectively. Our results are achieved by a careful analysis of the pseudo inverse $D^\dagger$ of $D$. In particular, our argument bypasses truncation of the spectrum altogether. Second, we also derive a ``scale free" result where the jumps in a piecewise constant signal $\theta^*$ may be of arbitrary size as instantiated by bounds of the order of $\|D\theta^*\|_0/n$. Finally, in the spirit of~\cite{DalHebLed14}, our results are expressed in terms of oracle inequalities. It gives us the ability to handle approximation error and ultimately prove adaptive and near optimal rates in several nonparametric regression models. These applications are detailed in section~\ref{sec:applications}. The scale-free results are key in obtaining optimal rates for cartoon images in subsection~\ref{sub:cartoon_functons}. Both the oracle part and the scale-free results are entirely novel compared to~\cite{mammen_locally_1997} and~\cite{WanShaSmo15}.

%
%Our fast rates lead to optimal rates of estimation in several models of interest for the signal \( \theta^\ast \) that were not obtained in~\cite{WanShaSmo15}. 

%Finally, our results are also valid in the case of misspecified models and are presented in terms of sharp oracle inequalities. Moreover, in the spirit of~\cite{DalHebLed14}, we allow for a trade-off between rates that depend on the $\ell_1$-norm of the discrete gradient and the number of active edges  ($\ell_0$-norm of the discrete gradient). The latter result allows for scale-independent rates in the case that the jumps across active edges are large. The scale-free results are key in obtaining optimal rates for cartoon images in subsection~\ref{sub:cartoon_functons}. Both the oracle part and the scale-free results are novel compared to~\cite{WanShaSmo15}.

Another step towards understanding the behavior of the TV denoiser was made in \cite{NeeWar13a,NeeWar13b}, where the authors focus on the case where the noise $\eps$ has small $\ell_2$ norm but is otherwise arbitrary as opposed to Gaussian in the present paper. This framework is fairly common in the literature on noisy compressed sensing. These results often do not translate directly the Gaussian noise setting since  properties of $\eps$ other than its $\ell_2$ norm are employed.
Nevertheless, one of their key lemmas, \cite[Proposition 7]{NeeWar13a}, provides additional insight into the relationship between Haar wavelet thresholding and total variation regularization.
In particular, it allows to prove that thresholding in the Haar wavelet basis attains rates comparable to the one we obtain for TV denoising, and it also can be used to prove the rates for TV denoising itself, albeit with an additional log factor.
We include these results in Appendix \ref{sec:haar-thresholding}.

\section{Sharp oracle inequality} % {{{2
\label{sec:sharp}

We start by defining two quantities involved in estimating the performance of the Lasso.

\begin{definition}[Compatibility factor, inverse scaling factor]
Let \( D \) be an incidence matrix, \( D \in \{-1,0,1\}^{m \times n} \), and write $S:=D^\dag =[s_1, \dots, s_m]$. The \emph{compatibility factor} of $D$ for a set \( T \subseteq [m] \) is defined as
	\label{def:comp-fac-rec-coeff}
	\begin{equation}
		\label{eq:v}
		\kappa_\emptyset := 1, \quad 
		\kappa_T = \kappa_T(D) := \inf_{\substack{\theta \in \R^n}} \frac{\sqrt{|T|} \| \theta \|_2}{\| (D \theta)_T \|_1} \quad \text{ for } T \neq \emptyset\,.
	\end{equation}
If we omit the subscript, then we mean the worst possible value of the constant, \ie \( \kappa =  \inf_{T \subseteq [m]} \kappa_T \).

Moreover, the  \emph{inverse scaling factor} of $D$ is defined as
	\begin{equation}
	\label{eq:cj}
		\rho = \rho(D) := \max_{j \in [m]} \| s_j \|_2\,.
	\end{equation}
\end{definition}

We prove the following main result.

\begin{theorem}[Sharp oracle inequality for TV denoising]
	\label{thm:Lasso-rates}
Fix  \( \delta \in (0,1) \), $T \subset [m]$ and let \( D \) being the incidence matrix of a connected graph \( G \). Define the regularization parameter
	\begin{equation}
		\label{eq:x}
		\lambda := \frac{1}{n} \sigma \rho \sqrt{2 \log\big(\frac{em}{\delta}\big)},
	\end{equation}
With this choice of $\lambda$, the TV denoiser $\hat \theta$ defined in \eqref{eq:y}  satisfies
	\begin{align}
		 \frac{1}{n} \| \hat\theta - \theta^\ast \|^2
		\leq {} & \inf_{\bar{\theta} \in \R^n} \left\{ \frac{1}{n} \| \bar{\theta} - \theta^\ast \|^2 + 4 \lambda \| (D \bar{\theta})_{T^c} \|_1 \right\}\\
	 	&	+ \frac{8 \sigma^2}{n} \left( \frac{|T| \rho^2}{\kappa_T^{2}} \log \big(\frac{em}{\delta}\big) + \log\big(\frac{e}{\delta}\big) \right).
		\label{eq:u}
	\end{align}
	on the estimation error with probability at least \( 1 - 2 \delta \). 
\end{theorem}
We delay the proof to the Appendix, Subsection~\ref{proof:main}.

The sharp oracle inequality \eqref{eq:u} allows trading off \( |T| \) with \( \| (D \bar{\theta})_{T^c} \|_1 \).
For \( T = \operatorname{supp}(D \bar{\theta}) \), we recover the $\ell_0$ rate \( \sigma^2 \kappa_T^{-2} \rho^2 \log(m/\delta) |T|/n \), while setting \( T \) to be the empty set, \( T = \emptyset \), yields the $\ell_1$ rate \( \sigma \rho \sqrt{\log (m/\delta)} \| D \bar{\theta} \|_1/n \).
We will see in Section \ref{sec:applications} that both rates are essential to get minimax rates for certain complexity classes

In order to evaluate the performance of the TV denoiser $\hat \theta$ on any graph $G$ and in particular on the 2D grid, we need   estimates on \( \rho \) and \( \kappa \). 

It turns out that bounding the compatibility factor is rather easy for all bounded degree graphs.

\begin{lemma}
	\label{lem:a}
	Let \( D \) be the incidence matrix of a graph \( G \) with maximal degree \( d \) and \( \emptyset \neq T \subseteq E \).
	Then,
	\begin{equation}
		\label{eq:h}
		\kappa_T = \inf_{\theta \in \R^n} \frac{\sqrt{|T|} \| \theta \|}{\| (D \theta)_T \|_1} \geq \frac{1}{2 \min\{\sqrt{d}, \sqrt{|T|}\}}.
	\end{equation}
\end{lemma}

\begin{proof}
	Let \( D \) be the incidence matrix of a graph \( G = (V,E) \), \( \theta \in \R^n \), and let $T \subset E =[m]$.
Moreover, denote by  \( d_i = \# \{ j \in [n] : (i, j) \in E \} \) the degree of vertex \( i \) and by \( d = \max_{i \in [n]} d_i \) the maximum degree of the graph.

	Then, by triangle inequality,
	\begin{equation*}
		\label{eq:d}
		\| (D \theta)_T \|_1
		% = \sum_{(i,j) \in T} | \theta_i - \theta_j |
		\le \sqrt{|T|} \sqrt{\sum_{(i,j) \in T} |\theta_i- \theta_j|^2 } 
%		\leq 2 \sqrt{|T|} \sqrt{\sum_{i \in [n]} d_i \theta_i^2 }\\
%		\leq {} & 
		\leq 2 \sqrt{|T|} \min \{\sqrt{|T|}, \sqrt{d}\} \| \theta \|_2
		% \le 2 \sqrt{d|T|} \| \theta \|_2.
		% % \qedhere
	\end{equation*}
%	where in line \eqref{eq:aa}, we used the Cauchy-Schwarz inequality on vectors of length \( |E| = m \), noticing that each entry \( \theta_i \) can appear at most \( d_i \) times.
\end{proof}

\section{Total variation regularization on the grid} % {{{3
\label{sec:tv-regularization-grid}

\subsection{TV regularization in 2D}
\label{sub:tv-2d}

In this section, we show that  \( \rho \lesssim \sqrt{\log n} \). Note that this is different from the 1D case: if we consider the incidence matrix \( \widetilde{D} \) of the path graph and for simplification add an additional row penalizing the absolute value of the first entry, \ie
\begin{equation}
	\label{eq:at}
	(\widetilde{D} \theta)_1 = \theta_1, \quad ( \widetilde{D} \theta)_i = \theta_{i} - \theta_{i - 1}, \quad i = 2, \dots, n,
\end{equation}
then one can show that \( (D^\dag)_{i,j} = (D^{-1})_{i,j} = \1(i \geq j) \).
Hence, in this case \( \rho = \sqrt{n} \). Moreover, the inverse scaling factor $\rho$ remains of the order $\sqrt{n}$ even if we close the path into a cycle.
The analyses of~\cite{WanShaSmo15} and~\cite{DalHebLed14} are geared towards refining the estimates used in the proof of Theorem~\ref{thm:Lasso-rates} in order to recover rates faster than \( n^{-1/2} \). Rather, we focus on extending results to the central example of the two dimensional grid, which is paramount in image processing.

We proceed to estimate \( \rho \) in the case of the total variation regularization  on the $N \times N$ 2D grid. Let \( n = N^2 \) and write \( D_1 \in \R^{ (N-1) \times N} \) for the incidence matrix of the path graph on $N$ vertices, \( D_1 x = x_{j+1} - x_{j} \), \( j = 1, \dots, N-1 \) for \( x \in \R^{ N } \).

Reshaping a signal \( \theta \) on the \( N \times N \) square in column major form as a vector \( \theta \in \R^n \), we can write the incidence matrix of the grid as
\begin{equation}
	\label{eq:i}
	D_2 =
	\begin{bmatrix}
		D_1 \otimes I\\
		I \otimes D_1
	\end{bmatrix}\,.
\end{equation}

\begin{proposition}
	\label{prp:tv-cols-2d}
The incidence matrix \( D_2 \) of the 2D grid on $n$ vertices has inverse scaling factor $\rho\lesssim\sqrt{\log n}$.
\end{proposition}
We delay the proof to the Appendix, Subsection~\ref{proof:2d}.
By combining the estimates from Lemma \ref{lem:a} and Proposition \ref{prp:tv-cols-2d} with Theorem \ref{thm:Lasso-rates}, we get the following rate for TV regularization on a regular grid in 2D.

\begin{corollary}
\label{cor:oracle-tv-2d}
Fix $\delta \in (0,1)$ and let \( D \) denote the incidence matrix of the 2D grid. Then there exist constants \( C, c > 0 \) such that the TV denoiser $\hat \theta$ with \( \lambda = c \sigma \sqrt{(\log n) \log (e n/\delta)}/n \) defined in  \eqref{eq:y} satisfies
	\begin{align}
		\label{eq:z}
		\frac{1}{n} \| \hat\theta - \theta^\ast \|^2 
		\leq {} & \inf_{\substack{\bar{\theta} \in \R^n\\ T \subseteq [m] }} \left\{ \frac{1}{n} \| \bar{\theta} - \theta^\ast \|^2 + 4 \lambda \| (D \bar{\theta})_{T^c} \|_1 \right\}
	  + \frac{C \sigma^2}{n} \left(|T|(\log n) \log (e n/\delta) + \log (e/\delta)\right),
	\end{align}	
	with probability at least \( 1 - 2 \delta \). In particular, it yields
	$$
	\mse(\hat \theta) \lesssim  \frac{ \sigma\|D\theta^*\|_1\wedge \sigma^2\|D\theta^*\|_0}{n}\log^2 (en/\delta)
	$$
	where $\|D\theta^*\|_0$ denotes the number of nonzero components of $D\theta^*$.
\end{corollary}

\subsection{TV regularization in higher dimensions} % {{{3
\label{sec:tv-regularization}

Akin to the 2D case, in \( d \) dimensions, we have \( n = N^d \) and we can write
\begin{equation}
	\label{eq:bf}
	D_d = \begin{bmatrix}
		D_1 \otimes I \otimes \dots \otimes I\\
				I \otimes D_1 \otimes \dots \otimes I\\
		\vdots\\
		I \otimes I \otimes \dots \otimes D_1
	\end{bmatrix}.
\end{equation}
Using similar calculations as in the 2D case, we can show that the inverse scaling factor \( \rho \) is now bounded by a constant, uniformly in $N$.

\begin{proposition}
	\label{prp:tv-cols-dd}
	For the incidence matrix of the regular grid on \( N^d \) nodes in \( d \) dimensions, \( \rho \leq C(d) \), for some \( C(d) > 0 \).
\end{proposition}
We delay the proof to the Appendix, subsection~\ref{proof:3d}.
It readily yields the following rate for TV regularization on a regular grid in \( d\ge 3 \) dimensions:

\begin{corollary}%[Oracle inequality for TV regularization]
\label{cor:oracle-tv}
Fix $\delta \in (0,1)$, an integer $d \ge 3$ and let \( D_d \) denote the incidence matrix of the $d$-dimensional grid. Then there exist constants \( C, c > 0 \) such that the TV denoiser $\hat \theta$ defined in  \eqref{eq:y} with \( \lambda = c \sigma \sqrt{ \log (e n/\delta)}/n \) satisfies
	\begin{align}
		\label{eq:oihighdim}
		\frac{1}{n} \| \hat\theta - \theta^\ast \|^2
		\leq {} & \inf_{\substack{\bar{\theta} \in \R^n\\ T \subseteq [m] }} \left\{ \frac{1}{n} \| \bar{\theta} - \theta^\ast \|^2 + 4 \lambda \| (D_d \bar{\theta})_{T^c} \|_1 \right\}
		& + \frac{C \sigma^2}{n} \left(|T| \log (e n/\delta) + \log (e/\delta)\right),
	\end{align}	
	with probability at least \( 1 - 2 \delta \). In particular, it yields
	$$
	\mse(\hat \theta) \lesssim  \frac{ \sigma\|D_d\theta^*\|_1\wedge \sigma^2\|D_d\theta^*\|_0}{n}\log (en/\delta)\,.
	$$
\end{corollary}

\subsection{The hypercube}
\label{sub:hypercube}

We note that in the case \( N = 2 \), the grid becomes the \( d \)-dimensional hypercube.
In this case, we can refine our analysis in this case to get the same result as in Proposition \ref{prp:tv-cols-dd} without dependence on the dimension.

\begin{proposition}
	\label{prp:hypercube}
For any $d \ge 1$, the inverse scaling factor associated to the $d$-dimensional hypercube satisfies \( \rho \leq 1 \).
\end{proposition}

\begin{proof}
  We use the same analysis as in the proof of Proposition \ref{prp:tv-cols-dd}, Subsection \ref{proof:3d}, noting that the eigenvectors of the 1-dimensional hypercube are given by \( v_1 = \begin{bmatrix}1& 1\end{bmatrix}^\top/\sqrt{2} \) and \( v_2 = \begin{bmatrix}1 & -1\end{bmatrix}^\top/\sqrt{2} \) with associated eigenvalues \( 0 \) and \( 2 \), respectively.
	Using the same notation as before, we have
	\begin{equation}
		\label{eq:cq}
		\langle v_{k_j} , e_{i_j} \rangle^2 \leq 1/2, \quad \text{for } \bm{k}, \bm{i} \in \{0, 1\}^d, j \in [d], \quad
		\langle v_{k_1} , d_{i_1} \rangle^2 \leq 2.
	\end{equation}
	This gives
	\begin{align}
		\label{eq:cm}
		\| s_{ \bm{i}}^{(1)} \|_2^2
		= {} & \sum_{\substack{\bm{k} \in \{0,1\}^d \setminus \{ 0 \}}} \left(\sum_{j=1}^{d} \lambda_{k_j}\right)^{-2} \langle v_{k_1} , d_{i_1} \rangle^2 \prod_{j = 2}^{d} \langle v_{k_j} , e_{i_j} \rangle^2\\
	\leq {} & 2^{2-d} \sum_{ \bm{k} \in \{0,1\}^d \setminus \{0\}} \left( \sum_{j=1}^{d} 2 k_j \right)^{-2} \leq 1.% \qedhere
	\end{align}
\end{proof}

\begin{corollary}%[Oracle inequality for TV regularization]
\label{cor:oracle-cube}
Fix $\delta \in (0,1)$, an integer $d \ge 1$ and let \( D_\square \) denote the incidence matrix of the $d$-dimensional hypercube. Then there exist constants \( C, c > 0 \) such that the TV denoiser $\hat \theta$ defined in  \eqref{eq:y} with $n=2^d$ and \( \lambda = c \sigma \sqrt{ \log (e n/\delta)}/n \) satisfies
	\begin{align}
		\label{eq:oicube}
		\frac{1}{n} \| \hat\theta - \theta^\ast \|^2 
	 	\leq {} & \inf_{\substack{\bar{\theta} \in \R^n\\ T \subseteq [m] }} \left\{ \frac{1}{n} \| \bar{\theta} - \theta^\ast \|^2 + 4 \lambda \| (D_\square \bar{\theta})_{T^c} \|_1 \right\}
		+ \frac{C \sigma^2}{n} \left(d |T| \log (e n/\delta) + \log (e/\delta)\right),
	\end{align}	
	with probability at least \( 1 - 2 \delta \). In particular, it yields
	$$
	\mse(\hat \theta) \lesssim  \frac{ \sigma\|D_\square\theta^*\|_1\wedge \sigma^2 d \|D_\square\theta^*\|_0}{n}\log (en/\delta)\,.
	$$
\end{corollary}

\section{Other graphs}
\label{sec:other-graphs}

\subsection{Complete graph} % {{{3
\label{sub:complete_graph}

Considering jumps along the complete graph has been proposed as a way to regularize when there is no actual structural prior information available; see \cite{She10} where it has been studied under the name \emph{clustered Lasso}.

\begin{proposition}
	\label{prp:comp-graph}
	For the complete graph \( K_n \), we have \( \kappa \gtrsim 1/\sqrt{n} \) and \( \rho \lesssim 1/n \).
\end{proposition}

\begin{proof}
  The bound on \( \kappa \) follows from Lemma \ref{lem:a}.
	To bound \( \rho \), note that we can write the pseudoinverse of the incidence matrix as
	\begin{equation}
		\label{eq:ch}
		S = D^\dag = (D^\top D)^\dag D^\top.
	\end{equation}
	The matrix \( D^\top D \) is the graph Laplacian of the complete graph which has the form \( n I - \bone \bone^\top \) from which we can read off its eigenvalues as \( \lambda_1 = 0 \), \( \lambda_i = n \), for \( i = 2, \dots, n \).
	Choose an eigenbasis \( \{ v_i \}_{i = 1, \dots, n} \) for \( D^\top D \).
	% Moreover, we can choose an eigenbasis \( \{v_i\}_{i = 1, \dots, n} \) that is delocalized in the sense that \( \| v^{(i)} \|_\infty \lesssim 1/\sqrt{n} \), for example the Fourier basis \( v^{(k)}_j = e^{2 \pi \ic k j/n}/\sqrt{n} \) because \( D^\top D \) is a circulant matrix.
	Then,
	\begin{equation}
		\label{eq:ci}
		\| s_j \|_2^2 = \sum_{k = 2}^{n} \frac{1}{\lambda_k^2} \langle v_k , d_j \rangle^2
		= \frac{1}{n^2} \sum_{k=2}^{n} \langle v_k , d_j \rangle^2 \leq \frac{1}{n^2} \| d_j \|_2^2 \leq \frac{2}{n^2}\,,
		% \lesssim \frac{1}{n^2} \frac{n}{n} = \frac{1}{n^2},
	\end{equation}
	for all \( j \).
\end{proof}
It yields the following corollary.
\begin{corollary}%[Oracle inequality for TV regularization]
\label{cor:oracle-complete}
Fix $\delta \in (0,1)$,  and let \( D_{K_n} \) denote the incidence matrix of the complete graph on $n$ vertices. Then there exist constants \( C, c > 0 \) such that the TV denoiser $\hat \theta$ defined in  \eqref{eq:y} with  \( \lambda = c \sigma \sqrt{ \log (e n/\delta)}/n^2 \) satisfies
	\begin{align}
		\label{eq:oicomplete}
		\frac{1}{n} \| \hat\theta - \theta^\ast \|^2
		\leq {} & \inf_{\substack{\bar{\theta} \in \R^n\\ T \subseteq [m] }} \left\{ \frac{1}{n} \| \bar{\theta} - \theta^\ast \|^2 + 4 \lambda \| (D_{K_n} \bar{\theta})_{T^c} \|_1 \right\}
		+ \frac{C \sigma^2}{n^2} \left(|T| \log (e n/\delta) + \log (e/\delta)\right),
	\end{align}	
	with probability at least \( 1 - 2 \delta \). In particular, it yields
	$$
	\mse(\hat \theta) \lesssim  \frac{ \sigma\|D_{K_n}\theta^*\|_1\wedge \sigma^2\|D_{K_n}\theta^*\|_0}{n^2}\log (en/\delta)\,.
	$$
\end{corollary}
	This implies that up to log factors, one performance bound on the TV denoiser for the clique is of the order \( |T|/n^2 \), where \( |T| \) is the number of edges with a jump in the ground truth \( \theta^\ast \).
	In the case of a signal that takes on \( k \ll n\) different values, with \( k-1 \) of them attained on small islands of size \( l \ll n \), this leads to a rate of \( k l/n \), the same we would get for the Lasso if the background value on the complement of the islands was zero.

	On the other hand, if there are two large components with different values, \( |T| \) will be of the order of \( n^2 \), so the result is not informative in this case.

\subsection{Star graph} % {{{3
\label{sub:star_graph}

Denote by \( S_n \) the star graph on \( n \) nodes, having one center node that is connected to \( n-1 \) leaves.
Note that the question of sparsistency of TV denoising for this graph, together with related ones, has been studied in \cite{OllVia15} as a way to regularize stratified data.

\begin{proposition}
	\label{prp:star-graph}
	For the star graph \( S_n \), we have \( \kappa_T \gtrsim 1/\sqrt{|T|} \) and \( \rho \leq 1 \).
\end{proposition}

\begin{proof}
	The estimate on \( \kappa \) follows directly from Lemma \ref{lem:a}. To compute \( \rho \), observe that
	\begin{equation}
		\label{eq:co}
		d_{i, j} = \left\{
		\begin{aligned}
			1, \quad {} & j = 1,\\
			-1, \quad {} & i = j - 1 \geq 2,\\
			0, \quad {} & \text{otherwise},
		\end{aligned}
		\right.\qquad
		s_{i, j} = D^\dag_{i,j} = \left\{
		\begin{aligned}
			- \frac{n-1}{n}, \quad {} & i = j + 1,\\
			\frac{1}{n}, \quad {} & \text{otherwise},
		\end{aligned}
		\right.
	\end{equation}
	whence the properties of the pseudoinverse can be verified by direct calculation.
	% \begin{align}
	% 	\label{eq:cn}
	% 	(DS)_{i,j} {} & = \sum_{k = 1}^{n} d_{i, k} s_{k, j} = \frac{1}{n} + \delta_{i,j} \frac{n-1}{n} - (1 - \delta_{i, j}) \frac{1}{n} = \delta_{i, j},\\
	% 	(SD)_{i,j} {} & = \delta_{i, j} - \frac{1}{n},
	% \end{align}
	% corresponding to the projection to the orthogonal complement of \( \operatorname{span}(\bone) \), \( I - \frac{1}{n} \bone \bone^T \).
	From this, we can estimate the norm of the columns of \( S \) by
	\begin{equation*}
		\label{eq:cp}
		\left\| s_j \right\|^2 = \sum_{i = 1}^{n-1} \frac{1}{n^2} + \left( \frac{n-1}{n} \right)^2 = \frac{n^2 - n}{n^2} \leq 1.% \qedhere
	\end{equation*}
\end{proof}

The following corollary immediately follows.
\begin{corollary}%[Oracle inequality for TV regularization]
\label{cor:oracle-star}
Fix $\delta \in (0,1)$,  and let \( D_{\star} \) denote the incidence matrix of the star graph on $n$ vertices. Then there exist constants \( C, c > 0 \) such that the TV denoiser $\hat \theta$ defined in  \eqref{eq:y} with  \( \lambda = c \sigma \sqrt{ \log (e n/\delta)}/n \) satisfies
	\begin{align}
		\label{eq:oistar}
		\frac{1}{n} \| \hat\theta - \theta^\ast \|^2
		\leq {} & \inf_{\substack{\bar{\theta} \in \R^n\\ T \subseteq [m] }} \left\{ \frac{1}{n} \| \bar{\theta} - \theta^\ast \|^2 + 4 \lambda \| (D_{\star} \bar{\theta})_{T^c} \|_1 \right\}
		+ \frac{C \sigma^2}{n} \left(|T|^2 \log (e n/\delta) + \log (e/\delta)\right),
	\end{align}	
	with probability at least \( 1 - 2 \delta \). In particular, it yields
	$$
	\mse(\hat \theta) \lesssim  \frac{ \sigma\|D_{\star}\theta^*\|_1\wedge \sigma^2\|D_{\star}\theta^*\|_0^2}{n}\log (en/\delta)\,.
	$$
\end{corollary}

The star graph leads to a useful regularization when most of the outer nodes take the same value as the central node and only a few outer nodes take a different one. Specifically, let $1$ denote the central vertex and consider the set $\Theta^\star(s) \subset \R_n$ defined for any integer $s \in [n-1]$ by
$$
\Theta^\star(s)=\Big\{ \theta \in \R^n\,:\, \sum_{j=2}^n \1(\theta_j \neq \theta_1)\le s\Big\}\,.
$$
Then it holds that
$$
\sup_{\theta^* \in \Theta^\star(s)}\frac{1}{n} \| \hat\theta - \theta^\ast \|^2 \lesssim\frac{\sigma^2s^2}{n}\log(en/\delta)
$$
	with probability at least \( 1 - 2 \delta \).
\subsection{Random graphs} % {{{3
\label{sub:random-graphs}

In the case of random graphs, it was noted in \cite{WanShaSmo15} that one can bound \( \rho \) if one has bounds on the second smallest eigenvalue of the Laplacian of the graph.
We can slightly improve on their estimation of \( \rho \).

\begin{proposition}
	\label{prp:deloc-gap-graphs}
	Suppose \( G \) is a connected graph whose Laplacian admits an eigenvalue decomposition \( D^\top D = V \Lambda V^\top \), with \( \Lambda = \operatorname{diag}(\lambda_1, \dots, \lambda_n) \), \( V = [ v_1, \dots, v_n]  \), \( 0 = \lambda_1 \leq \lambda_2 \leq \dots \leq \lambda_n \).

	If the graph Laplacian has a \emph{spectral gap}, \ie there exists a constant \( c_1 > 0 \) such that \( \lambda_2 \geq c_1 \),
	then \( \rho \le \sqrt{2}/c_1\).
%	\begin{equation}
%		\label{eq:ay}
%		\rho \leq \frac{\sqrt{2}C_1}{c_1}.
%	\end{equation}
\end{proposition}
\begin{proof}
	Writing \( D^\dag = [s_1, \dots, s_m] = (D^\top D)^\dag D^\top \), \( D^\top = [d_1, \dots, d_m] \), we note that the columns \( d_j \) have 2-norm \( \|d_j\|_2 = \sqrt{2} \) because \( D \) is the incidence matrix of a graph, so
  \begin{equation*}
  	\label{eq:ax}
		\| s_j \|_2^2
		= \sum_{k = 2}^{n} \frac{1}{\lambda_k^2} \langle v_{k} , d_j \rangle^2 \leq \frac{1}{\lambda_2^2} \sum_{k = 2}^{n} \langle v_k , d_j \rangle^2 \leq \frac{1}{\lambda_2^2} \left\| d_j \right\|_2^2 \leq \frac{2}{c_1^2}.% \qedhere
  \end{equation*}
\end{proof}

We can combine this with bounds on the spectral gap of two families of random graphs, Erd\H{o}s-R\'enyi random graphs \( G(n,p) \) and random regular graphs.
Both of these models exhibit a spectral gap of the order \( O(d) \) in a regime where the degree increases logarithmically with the number of vertices, see \cite{KolOstVon14} and \cite{Fri04}, respectively.
Together with the bound on \( \kappa \) from Lemma \ref{lem:a}, we get the following rate.

\begin{corollary}
\label{cor:er}
Fix $\delta \in (0,1)$ and let \( D \) denote the incidence matrix of either a random $d$-regular graph with $d_n=d_0(\log n)^{\beta}$ or Erd\H{o}s-R\'enyi random graph with \( G(n,p) \) with \( p_n=d_n/n \) for some constant $\beta>0$ and \( d_0 > 1 \). Then there exist constants \( C, c > 0 \) such that the TV denoiser $\hat \theta$ defined in  \eqref{eq:y} with \( \lambda = c \sigma \sqrt{\log (e d_n n/\delta)}/(d_n n) \) satisfies
	\begin{align}
		\frac{1}{n} \| \hat\theta - \theta^\ast \|^2
		\leq {} & \inf_{\substack{\bar{\theta} \in \R^n\\ T \subseteq [m] }} \left\{ \frac{1}{n} \| \bar{\theta} - \theta^\ast \|^2 + 4 \lambda \| (D \bar{\theta})_{T^c} \|_1 \right\}
		+ \frac{C \sigma^2}{d_n n} \left(|T| \log (e n/\delta) + \log (e/\delta)\right),
	\end{align}	
	with probability at least \( 1 - 2 \delta \) over \( \varepsilon \) and with high probability over the realizations of the graph. In particular, it yields
	$$
	\mse(\hat \theta) \lesssim \frac{\sigma^2\|D\theta^*\|_0}{d_n n}\log (en/\delta)\,,
	$$
	where $\|D\theta^*\|_0$ denotes the number of nonzero components of $D\theta^*$.
\end{corollary}

In the context of TV denoising, Erd\H{o}s-R\'enyi random graphs with expected degree \( d_n \) can be considered a sparsification of the complete graph considered in Section \ref{sub:complete_graph}.
In the same model considered in Subsection~\ref{sub:complete_graph} of \( k \) islands with \( l \) nodes each, we would get a performance rate of \( kl/n \), the same as before.
On the other hand, the underlying graph is much sparser, so we could possibly get a computational benefit from choosing it instead of the complete graph. The behavior of random graphs is compared to that of the complete graph in Section~\ref{sub:numerical_experiments} in the appendix. They indicate that the computational saving occur at a negligible statistical cost.

\subsection{Power graph of the cycle} % {{{3
\label{sub:power_graph_of_the_cycle}

In practice, nearest neighbor graphs often arise in the context of spatial regularization. The grid is one such example and as an extension, we consider the $k$th power of the cycle graph as a toy example to study the effect of increasing the connectivity of the graph.

Define the cycle graph \( C_n \) to be the graph on \( n \) vertices with \( i \sim j \) if and only if \( i - j \equiv \pm 1 \bmod n \) and its \( k \)th power graph \( C_n^k \) as the graph with the same vertex set but with \( i \sim j \) if and only if there is a path of length at most \( k \) from \( i \) to \( j \) in \( C_n \).

\begin{proposition}
	\label{prp:power-graph}
	For \( G = C_n^k \) where \( k \leq n/2 \),  \( \rho \lesssim \sqrt{n}/k^3 + 1 \) and \( \kappa \gtrsim 1/\sqrt{k} \).
\end{proposition}

We delay the proof to the Appendix, Subsection \ref{sec:proof-power-graph}.

\begin{corollary}
\label{cor:power-graph}
Fix $\delta \in (0,1)$ and let \( D \) denote the incidence matrix of \( C_n^k \).
Then there exist constants \( C, c > 0 \) such that the TV denoiser $\hat \theta$ defined in  \eqref{eq:y} with \( \lambda = c \sigma \sqrt{\log (e n/\delta)}/(\sqrt{n} k^3 \wedge n) \) satisfies
	\begin{align}
		\frac{1}{n} \| \hat\theta - \theta^\ast \|^2
		\leq {} & \inf_{\substack{\bar{\theta} \in \R^n\\ T \subseteq [m] }} \left\{ \frac{1}{n} \| \bar{\theta} - \theta^\ast \|^2 + 4 \lambda \| (D \bar{\theta})_{T^c} \|_1 \right\}\\
		&+ C \sigma^2 \left(\frac{1}{k^5} \vee \frac{k}{n} \right)\left(|T| \log (e n/\delta) + \log (e/\delta)\right),
		\label{eq:de}
	\end{align}	
	with probability at least \( 1 - 2 \delta \).
	In particular, it yields
	$$
	\mse(\hat \theta) \lesssim \sigma^2 \| D \theta^\ast \|_0 \left(\frac{1}{k^5} \vee \frac{k}{n} \right) \log (en/\delta)\,,
	$$
	where $\|D\theta^*\|_0$ denotes the number of nonzero components of $D\theta^*$.
\end{corollary}

\section{Applications to nonparametric regression} % {{{2
\label{sec:applications}

The rate for the grid obtained in Corollaries~\ref{cor:oracle-tv-2d} and~\ref{cor:oracle-tv} can be used to derive rates for nonparametric function estimation in dimension \( d \geq 2 \).
This allows us to  generalize the results \cite[Proposition~6]{DalHebLed14} for the adaptive estimation of H\"older functions and \cite{chatterjee_matrix_2015, Bel15} for the estimation of bi-isotonic matrices.
Moreover, we can also generalize to piecewise H\"older functions, called ``cartoon images".

In the first two subsections, we are interested in real valued functions on $[0,1]^d$.  To relate function estimation to our problem, consider the vectors \( \theta \) to be a discretization of a continuous signal \( f \colon [0,1]^d \to \R \) on the regular grid \( \mathcal{X}_N^d := \{x_{ \bm{i}} := \bm{i}/N : \bm{i} \in [N]^d \} \), so \( \theta_{ \bm{i}} = f( x_{\bm{i}}) = f(i_1/N, \dots, i_d/N) \), \( \bm{i} \in [N]^d \). Furthermore, for any function, $f\colon [0,1]^d \to \R$, define the pseudo-norm $\|f\|_n$ by
$$
\|f\|_n^2=\frac{1}{n}\sum_{\bm{i} \in [N]^d}f(x_{ \bm{i}})^2\,.
$$

\subsection{H\"older functions} % {{{3
\label{sub:holder_functions}

In \cite[Proposition 7]{DalHebLed14}, the authors showed that the TV denoiser in one dimension achieves the minimax rate \( n^{-\frac{2 \alpha}{2 \alpha+1}} \) for estimating H\"older continuous functions (with parameter $\alpha \in (0,1]$) on a bounded interval, up to logarithmic factors.
Here, we show that the TV denoiser achieves a rate of \( n^{-\frac{2 \alpha}{d \alpha + d}} \), again up to logarithmic factors, which means it is near minimax for two-dimensional observations as well. Unlike the one-dimensional result of~\cite{DalHebLed14}, the TV denoiser in two dimensions is adaptive to the unknown parameter $\alpha$.

\begin{definition}[H\"older function]
\label{def:holder-fcns}
	For \( \alpha \in (0,1] \), \( L > 0 \), we say that a function \( f \colon [0,1]^d \to \R \)	is \emph{\( (\alpha, L) \)-H\"older continuous} if it satisfies
	\begin{equation}
		\label{eq:bv}
		| f(y) - f(x) | \leq L \| x - y \|_{\infty}^\alpha \quad \text{for all } x, y \in [0,1]^d.
	\end{equation}
	For such an \( f \), we write \( f \in H(\alpha, L) \).
\end{definition}

Note that we picked the \( \ell_{\infty} \)-norm for convenience here.
By the equivalence of norms in finite dimensions, the \( \ell_2 \)-norm would yield the same definition up to a dimension-dependent constant.
Moreover, for samples of a H\"older continuous function on a grid, Definition \ref{def:holder-fcns} implies
\begin{equation}
	\label{eq:au}
	| \theta_{ \bm{i}} - \theta_{ \bm{j}} | \leq L N^{- \alpha} \| \bm{i} - \bm{j}\|_\infty^\alpha,
\end{equation}
so we can directly work with the \( \ell_\infty \)-distance between the indices.

\begin{proposition}
	\label{prp:holder-fcns}
	Fix \( \delta \in (0,1) \), \( d \geq 2 \), \( L > 0 \), $N \ge 1$, $n=N^d$ and \( \alpha \in (0,1] \) and let \( y \) be sampled according to the Gaussian sequence model \eqref{eq:al}, where \(\theta^*_{ \bm{i}} =  f^*(x_{\bm{i}}) \), \( \bm{i} \in [N]^d \) for some unknown function $f^*:[0,1]^d \to \R$.
	There exist positive constants $c$, $C$ and \( C' = C'(\sigma, L, d) \) such that the following holds.
Let $\hat \theta$ be the TV denoiser defined in~\eqref{eq:y} for the \( N^d \) grid with incidence matrix \( D_d \) and tuning parameter \( \lambda = c \sigma \sqrt{r_d(n) \log(e n/\delta)}/n, c>0 \) where $r_2(n)=\log n$ and $r_d (n)=1$ for $d \ge 3$. Moreover, let $\hat f:[0,1]^d \to \R$ be defined by $\hat f(x_{\bm{i}})=\hat \theta_{\bm{i}}$ for $\bm{i} \in [N]^d$ and arbitrarily elsewhere on the unit hypercube $[0,1]^d$.

Further, assume that \( N \geq C'(L, \sigma, d) \sqrt{r_d(n) \log(en/\delta)} \).
Then,
\begin{align}
		\label{eq:ab}
%		 \| \widehat{f} - f^\ast \|_n^2
%		\leq \inf_{ \bar{f} \in H(\alpha, L)} \left\{ \| \bar{f} - f^\ast \|^2_n \right\} + C L^2 \left( \frac{\sigma^2 \sqrt{r_d(n) \log(e n/\delta)}}{L^2 n^{2/d}} \right)^{\frac{2 \alpha}{2 \alpha + 2}} + C \frac{\sigma^2 r_d(n) \log(e n/\delta)}{n^{2/d}},
		 \| \widehat{f} - f^\ast \|_n^2
		\leq {} & \inf_{ \bar{f} \in H(\alpha, L)} \left\{ \| \bar{f} - f^\ast \|^2_n \right\}
		+ C \frac{\big(L^2(\sigma\sqrt{r_d(n) \log(e n/\delta)})^{2\alpha} \big)^{\frac{1}{\alpha+1}}}{n^{\frac{2 \alpha}{d \alpha + d}}} + C \frac{\sigma^2}{n}\log (e/ \delta)\,,
	\end{align}
	with probability at least \( 1 - 2 \delta \)\,. In particular, for $d=2$, it yields the near optimal rate
	\begin{align}
		\| \widehat{f} - f^\ast \|_n^2
		\leq {} & \inf_{ \bar{f} \in H(\alpha, L)} \left\{ \| \bar{f} - f^\ast \|^2_n \right\}
		+ C \big(L^2(\sigma \log(e n/\delta))^{2\alpha} \big)^{\frac{1}{\alpha+1}}n^{-\frac{2 \alpha}{2 \alpha + 2}}
		+ C \frac{\sigma^2}{n}\log (e/ \delta)\,.
	\end{align}
\end{proposition}

The proof of Proposition \ref{prp:holder-fcns} is deferred to the Appendix, Subsection \ref{sec:proof-holder-fcns}.

Unlike \cite[Proposition 7]{DalHebLed14}, this result does not require the knowledge of \( L \) or \( \alpha \) to compute the tuning parameter \( \lambda \), but only the noise level \( \sigma \). As a result, the estimator is therefore adaptive to the smoothness of the underlying function. This effect comes from better estimates on \( \rho \) than in the 1D case.

	\cite{mammen_locally_1997} have shown that asymptotically, TV regularization together with spline regression achieves the minimax rate for the estimation of \( k \)-times differentiable functions in 2D.
	Our result however holds for finite sample size and fractional smoothness, albeit only for \( \alpha \in (0,1] \).

It is not surprising that our results are suboptimal for $d\ge 3$. Indeed, penalizing by the size of jumps is not appropriate for H\"older functions. One should rather penalize by the number of blocks. It is merely a coincidence that in two dimensions this method leads to optimal and adaptive rates for H\"older functions.
%\end{remark}

\subsection{Piecewise constant and piecewise H\"older functions} % {{{3
\label{sub:cartoon_functons}

Recall that Cor\-ollary \ref{cor:oracle-tv-2d} allows us to get scale free results, i.e., bounds that do not scale with jump height. It is therefore well suited to detect sharp boundaries, one of the features often associated with total variation regularization.
To formalize this point, we analyze two models that involve a boundary, namely piecewise constant and piecewise smooth signals. The framework below largely builds upon \cite{willett_faster_2005}.

First, let us define the box-counting dimension of a set, which we will use as the measure of the complexity of the boundary.

\begin{definition}[Box-counting dimension]
	\label{def:box-counting-dim}
	Let \( B \subseteq [0,1]^d \) be a set and denote by \( N(r) \) the minimum number of (Euclidean) balls of radius \( r \) required to cover \( B \).
	The \emph{box-counting dimension} of \( B \) is defined as
	\begin{equation}
		\label{eq:bo}
		\operatorname{dim}_{\mathrm{box}}(B) := \limsup_{r \to 0} \frac{\log N(r)}{\log (1/r)}.
	\end{equation}
\end{definition}

The box-counting dimension generalizes the notion of linear dimension. 
For instance, if \( B \) is a smooth \( d_0 \)-dimensional manifold, then its box-counting dimension is equal to \( d_0 \).

\begin{definition}[Piecewise constant functions]
\label{def:pw-constant}
	For \( \beta > 0 \), we call a function \( f \colon \R^d \to \R \) \emph{piecewise constant} and write \( f \in PC(d,\beta) \) if there is an associated  boundary set \( B(f) \) such that:
	\begin{enumerate}
		\item The function \( f \) is locally constant on \( [0,1]^{d} \setminus B(f) \), \ie for all \( x \in [0,1]^d \setminus B(f) \), there is an \( \varepsilon > 0 \) such that for all \( y \) with \( \| y - x \| < \varepsilon \), \( f(x) = f(y) \).
		\item The boundary set \( B(f) \) has covering number \( N(r) \leq \beta r^{-(d-1)} \) for some $\beta>0$. In particular, $\operatorname{dim}_{\mathrm{box}}(B)\le d-1$.
	\end{enumerate}
\end{definition}

\begin{definition}[Piecewise H\"older functions]
\label{def:pw-holder}
	For \( \alpha \in (0,1] \) and \( \beta, L > 0 \), we call a function \( f \colon \R^d \to \R \) \emph{piecewise H\"older}, \( f \in PH(d, \beta, \alpha, L) \), if there is an associated boundary set \( B(f) \) such that:
	\begin{enumerate}
		\item \( f \) is locally \( L \)-H\"older on \( [0,1]^{d} \setminus B(f) \), \ie for all \( x \in [0,1]^d \setminus B(f) \), there is an \( \varepsilon > 0 \) such that for all \( y \) with \( \| y - x \| < \varepsilon \), \( | f(x) - f(y) | \leq L \| x - y \|_\infty^\alpha \).
		\item \( B(f) \) has box-counting dimension at most \( d-1 \), and its covering number \( N(r) \) is bounded by \( N(r) \leq \beta r^{-(d-1)} \).
	\end{enumerate}
\end{definition}

One intuition behind these definitions is to consider the signal as a ``cartoon image" containing large patches that are constant or fairly smooth, split by sharp boundaries.

Using Corollary \ref{cor:oracle-tv-2d}, we can now establish estimation rates for these classes of functions.

\begin{proposition}
	\label{prp:est-pw-cst}
		Fix \( \delta \in (0,1) \), \( d \geq 2 \),  $N \ge 1$, $n=N^d$  and let \( y \) be sampled according to the Gaussian sequence model \eqref{eq:al}, where \(\theta^*_{ \bm{i}} =  f^*(x_{\bm{i}}) \), \( \bm{i} \in [N]^d \) for some unknown function $f^*\in PC(d, \beta)$.
		There exist positive constants $c$ and $C$ such that the following holds.
Let $\hat \theta$ be the TV denoiser defined in~\eqref{eq:y} for the $d$-dimensional grid with incidence matrix \( D_d \) and tuning parameter \( \lambda = c \sigma \sqrt{r_d(n) \log(e n/\delta)}/n \), where $r_2(n)=\log n$ and $r_d (n)=1$ for $d \ge 3$.
Moreover, let $\hat f:[0,1]^d \to \R$ be defined by $\hat f(x_{\bm{i}})=\hat \theta_{\bm{i}}$ for $\bm{i} \in [N]^d$ and arbitrarily elsewhere on the unit hypercube $[0,1]^d$.
Then,
%
%
%	Let \( \beta > 0 \) and \( f^\ast \in PC(\beta) \) be piecewise constant.
%	Then, there are constants \( C = C(d), c = c(d) > 0 \) such that for any tolerance level \( \delta \in (0,1) \), with probability at least \( 1 - 2 \delta \), the TV denoiser \( \widehat{f} \) on a regular grid with \( \lambda = c \sigma \sqrt{r_d(n) \log (e n/\delta)}\) achieves
	\begin{equation}
		\label{eq:br}
		\| \widehat{f} - f^\ast \|_n^2 \lesssim \frac{\sigma^2 \beta}{n^{1/d}} r_d(n) \log(e n/\delta) +  \frac{\sigma^2}{n}\log(e /\delta)\,,
	\end{equation}
	with probability at least \( 1 - 2 \delta \).
\end{proposition}

The proof of Proposition \ref{prp:est-pw-cst} is deferred to the Appendix, Subsection \ref{sec:proof-est-pw-cst}.

Combining the results for piecewise constant and H\"older smooth functions, we can get the following extension to piecewise smooth functions.
\begin{proposition}
	\label{prp:cartoon-fcns}
		Fix \( \delta \in (0,1) \), \( d \geq 2 \),  $N \ge 1$, $n=N^d$  and let \( y \) be sampled according to the Gaussian sequence model \eqref{eq:al}, where \(\theta^*_{ \bm{i}} =  f^*(x_{\bm{i}}) \), \( \bm{i} \in [N]^d \) for some unknown function $f^*\in PH(d,\beta, \alpha, L)$,  $\alpha \in (0,1]$, $L>0$, $\beta>0$.
		There exist positive constants $c$, $C$ and \( C' = C'(\sigma, L, d) \) such that the following holds.
		Let $\hat \theta$ be the TV denoiser defined in~\eqref{eq:y} for the $d$-dimensional grid with incidence matrix \( D_d \) and tuning parameter \( \lambda = c \sigma \sqrt{r_d(n) \log(e n/\delta)}/n \), where $r_2(n)=\log n$ and $r_d (n)=1$ for $d \ge 3$.
		Moreover, let $\hat f:[0,1]^d \to \R$ be defined by $\hat f(x_{\bm{i}})=\hat \theta_{\bm{i}}$ for $\bm{i} \in [N]^d$ and arbitrarily elsewhere on the unit hypercube $[0,1]^d$.

If \( N \geq C'(L, \sigma, d) \sqrt{r_d(n) \log(en/\delta)} \), then
%
%	Let \( d \geq 2 \), \( \alpha \in [0,1] \), \( \beta, L > 0 \), \( f^\ast \in PH(d,\beta, \alpha, L)\) be piecewise H\"older.
%	Then, there are constants \( C, c > 0 \) depending on \( d \) such that the fused Lasso estimator \eqref{eq:y} for the \( N^d \) grid \( \mathcal{X}_N^d \) with incidence matrix \( D_d \) and tuning parameter \( \lambda = \sigma c \sqrt{r_d(n) \log(n/\delta)} \) for any tolerance level \( \delta \in (0,1) \) has rate
	\label{prp:holder-rate}
	\begin{align}
		\label{eq:e}
		\frac{1}{n} \| \widehat{f} - f^\ast \|^2
		\lesssim {} & \frac{\big(L^2(\sigma\sqrt{r_d(n) \log(e n/\delta)})^{2\alpha} \big)^{\frac{1}{\alpha+1}}}{n^{\frac{2 \alpha}{d \alpha + d}}}
		+  \frac{\sigma^2\beta}{n^{1/d}}r_d(n) \log(en\delta)+\frac{\sigma^2}{n}\log(e /\delta)\,,
	\end{align}
	with probability at least \( 1 - 2 \delta \).
\end{proposition}

The proof of Proposition \ref{prp:cartoon-fcns} is deferred to the Appendix, Subsection \ref{sec:proof-cartoon-fcns}.

For a Lipschitz boundary in two dimensions, this matches the minimax bound \( n^{-2 \alpha /(2 \alpha + 2)} \vee n^{-1/2} \) for boundary fragments in \cite[Theorem 5.1.2]{KorTsy93} up to logarithmic factors.
However, unlike the framework of \cite{KorTsy93}, our techniques do not allow an improvement of the bound for smoother boundaries parametrization because \( |T| \) will always be of the order \( O(N^{d-1}) \).
On the other hand, unlike the algorithms in \cite{KorTsy93} and \cite{arias-castro_oracle_2012}, our analysis allows for any jump sizes, so TV regularization automatically adapts to both \( B(f) \) and \( \alpha \).

\subsection{Bi-isotonic matrices} % {{{3
\label{sub:monotone_functions}

In our final example, we consider two-di\-men\-sional signals that increase in both directions, sometimes referred to~\emph{bi-isotonic}. The class of bi-isotonic matrices is defined as follows,
\begin{equation}
	\label{eq:bc}
	\mathcal{M} := \big\{ \theta \in \R^{ N \times N } : \theta_{j_1, j_2} \geq \theta_{i_1, i_2} \text{ if } j_1 \geq i_1 \text{ and } j_2 \geq i_2 \big\}.
\end{equation}
Recently, it was shown in \cite{chatterjee_matrix_2015, Bel15} that the least squares estimator for \( \mathcal{M} \) yields the near minimax rate \( \sqrt{D(\theta^\ast)/n}\, (\log n)^4 \), where \( D(\theta^\ast) := (\theta^\ast_{ N, N} - \theta^\ast_{1,1})^2 \) denotes the square variation of the matrix.

In the following, we show that the 2D TV denoiser can match this rate and that it also improves on the exponent of the \( \log \) factors.

\begin{proposition}
	\label{prp:monotone-rate}
	Let \( y \) be a sample of the Gaussian sequence model \eqref{eq:al}, \( \delta \in (0,1) \) and denote by \( \theta^\uparrow \) the projection of \( \theta^\ast \) onto \( \mathcal{M} \).	Fix $\delta \in (0,1)$ and let  $\hat \theta$ denote the TV denoiser on the 2D grid with \( \lambda = c \sigma \sqrt{(\log n) \log (e n/\delta)}/n \) defined in  \eqref{eq:y}. Then there exists a constant $C>0$ such that
%	\begin{equation}
%		\label{eq:z}
%		\frac{1}{n} \| \hat\theta - \theta^\ast \|^2 \leq \inf_{\substack{\bar{\theta} \in \R^n\\ T \subseteq [m] }} \left\{ \frac{1}{n} \| \bar{\theta} - \theta^\ast \|^2 + 4 \lambda \| (D \bar{\theta})_{T^c} \|_1 \right\} + \frac{C \sigma^2}{n} \left(|T|(\log n) \log (e n/\delta) + \log (e/\delta)\right),
%	\end{equation}	
%	with probability at least \( 1 - 2 \delta \). In particular, it yields
%	$$
%	\mse(\hat \theta) \lesssim  \frac{ \sigma\|D\theta^*\|_1\wedge \sigma^2\|D\theta^*\|_0}{n}\log^2 (en/\delta)
%	$$
%	where $\|D\theta^*\|_0$ denotes the number of nonzero components of $D\theta^*$.
%
%
%	Let \( y \) be a sample of the Gaussian sequence model \eqref{eq:al}, \( \delta \in (0,1) \) and denote by \( \theta^\uparrow \) the projection of \( \theta^\ast \) onto \( \mathcal{M} \).
%	Then, there are constants \( C, c > 0 \) such that the fused Lasso estimator \eqref{eq:y} for the \( N \times N  \) grid with incidence matrix \( D_2 \) and tuning parameter \( \lambda = c \sigma \log N \sqrt{(\log n)\log(e n/\delta)} \) has rate
	\begin{equation}
		\label{eq:ac}
		% \frac{1}{n} \| \hat\theta - \theta^\ast \|^2
		% \leq \frac{1}{n} \| \theta^\uparrow - \theta^\ast \|^2 + C \sigma \sqrt{\frac{D(\theta^\uparrow)\vee 1}{n}}\log(e n/\delta)\,,
		\frac{1}{n} \| \hat\theta - \theta^\ast \|^2
		\leq \frac{1}{n} \| \theta^\uparrow - \theta^\ast \|^2 + C\sigma \sqrt{\frac{{(\log n)\log(n/\delta)}}{{n}} }\sqrt{D( \theta^\uparrow)}+ C\frac{\sigma^2}{n} \log(e /\delta),
	\end{equation}
	with probability at least \( 1 - 2 \delta \).
\end{proposition}

\begin{proof}
	We use the slow rate version of \eqref{eq:z} for \( \bar{\theta} = \theta^\uparrow \),
	\begin{equation}
		\label{eq:af}
		\frac{1}{n} \| \hat\theta - \theta^\ast \|^2
		\leq \frac{1}{n} \| \theta^\uparrow - \theta^\ast \|^2 + 4 \lambda \| D \theta^\uparrow \|_1 + \frac{C\sigma^2}{n} \log (e /\delta).
	\end{equation}
	Because \( \theta^\uparrow \) is bi-isotonic, summing along the rows yields
	\begin{align}
		\label{eq:f}
		\sum_{i = 1}^{N} \sum_{j = 1}^{N-1} | \theta^\uparrow_{i, j+1} - \theta^\uparrow_{i, j} |
		\leq {} & \sum_{i = 1}^{N} ( \theta^\uparrow_{i, N} - \theta^\uparrow_{i, 1} ) \leq N ( \theta^\uparrow_{N,N} - \theta^\uparrow_{1,1} ),
	\end{align}
	and similarly along columns, which combined gives us
	\begin{align}
		\label{eq:ak}
		\| D\theta^\uparrow \|_1 \leq 2 N ( \theta^\uparrow_{N, N} - \theta^\uparrow_{1, 1} ) = 2 \sqrt{n D( \theta^\uparrow )}.
	\end{align}
	Plugging this into \eqref{eq:af}, together with inserting the value of \( \lambda \), we have
	\begin{equation}
		\label{eq:bd}
		\frac{1}{n} \| \hat\theta - \theta^\ast \|^2
		\leq \frac{1}{n} \| \theta^\uparrow - \theta^\ast \|^2 + C\sigma \sqrt{\frac{{(\log n)\log(n/\delta)}}{{n}} }\sqrt{D( \theta^\uparrow)}+ C\frac{\sigma^2}{n} \log(e /\delta),
	\end{equation}
	for some \( C > 0 \).
\end{proof}
We recover the results of \cite{chatterjee_matrix_2015, Bel15} with a smaller exponent in the logarithmic factor.
On the other hand, the TV-denoiser requires an estimate for \( \sigma \) (or at least an upper bound), unlike the least squares estimator, which does not require any tuning.

Note further that our rate scales with $\sigma$ rather than $\sigma^2$ in \cite{chatterjee_matrix_2015, Bel15}. This is because we use a ``slow rate" bound. 

Unlike~\cite{chatterjee_matrix_2015, Bel15} we do not show that our estimator adapts to the number of rectangles on which the matrix is piecewise constant. In particular, they show that if the number of such rectangles is a constant, then the least squares estimator achieves a fast rate of order $\sigma^2(\log n)^8/n$. This is not the case in the present paper. Indeed, the TV denoiser is not the correct tool for that. Even in the case of two rectangles, the number of active edges on the 2D grid is already linear in $N$ leading to rates that are slower than $\sigma^2/N\gg \sigma^2(\log n)^8/n$. Nevertheless, it is not hard to show that if $\theta^*$ is an $N\times N$ matrix with a triangular structure the form $\theta^*_{ij}=\1(i \ge j)$, then this matrix is well approximated by $N$ rectangles. In this case, the results of~\cite{chatterjee_matrix_2015, Bel15} yield a bound for the least squares estimator $\hat \theta^{\textsc{ls}}$ of the form
$$
\frac{1}{n} \| \hat\theta^{\textsc{ls}} - \theta^\ast \|^2 \le C \sigma^2\frac{(\log n)^8}{\sqrt{n}}
$$
and it is not hard to see that the TV denoiser  yields
$$
\frac{1}{n} \| \hat\theta - \theta^\ast \|^2 \le C (\sigma\wedge 1)^2\frac{(\log n)^2}{\sqrt{n}}
$$
where both results are stated with large but constant probability (say 99\%). It is not excluded that the least squares estimator still achieves faster rates in this case but the currently available results do not lead to better rates.

Finally, note that unlike \cite[Proposition 6]{DalHebLed14}, \( \lambda \) does not have to depend on \( D( \theta^\uparrow) \) here because of the better behavior of \( \rho \) for the 2D grid.

\medskip

{\bf Acknowledgments}
% \acks{
We would like to thank Vivian Viallon for bringing the paper by \cite{ShaSinRin12} to our attention. We thank the participants in the workshop ``Computationally and Statistically Efficient Inference for Complex Large-scale Data", that took place in Oberwolfach on March 6--12, 2016; in particular Axel Munk for pointers to the literature on the spectral decomposition of the Toeplitz matrix in~\eqref{EQ:toeplitzmat} and Alessandro Rinaldo for interesting discussion. Finally, we thank Ryan Tibshirani for pointing us to the paper~\cite{WanShaSmo15} and discussing his results with us. 

Philippe Rigollet is supported in part by NSF grants DMS-1317308 and CAREER-DMS-1053987.
% }
%A short version of this paper was already submitted when we became aware of the prior work~\cite{WanShaSmo15}. We are grateful to Ryan for being forthcoming in dealing with this delicate situation and being willing to look at a preliminary version of the present paper to help us assess the differences between the two works.

\bibliographystyle{amsalpha}
\bibliography{TVRates_noabstracts}

\newcommand{\etalchar}[1]{$^{#1}$}
\providecommand{\bysame}{\leavevmode\hbox to3em{\hrulefill}\thinspace}
\providecommand{\MR}{\relax\ifhmode\unskip\space\fi MR }
% \MRhref is called by the amsart/book/proc definition of \MR.
\providecommand{\MRhref}[2]{%
  \href{http://www.ams.org/mathscinet-getitem?mr=#1}{#2}
}
\providecommand{\href}[2]{#2}
\begin{thebibliography}{VLLHP16}

\bibitem[ACSW12]{arias-castro_oracle_2012}
Ery Arias-Castro, Joseph Salmon, and Rebecca Willett, \emph{Oracle inequalities
  and minimax rates for nonlocal means and related adaptive kernel-based
  methods}, SIAM Journal on Imaging Sciences \textbf{5} (2012), no.~3,
  944--992.

\bibitem[AT16]{ArnTib16}
Taylor~B. Arnold and Ryan~J. Tibshirani, \emph{Efficient implementations of the
  generalized lasso dual path algorithm}, Journal of Computational and
  Graphical Statistics \textbf{25} (2016), no.~1, 1--27.

\bibitem[Bel15]{Bel15}
Pierre~C. Bellec, \emph{Sharp oracle inequalities for {{Least Squares}}
  estimators in shape restricted regression}, arXiv preprint arXiv:1510.08029
  (2015).

\bibitem[BLM13]{BouLugMas13}
St{\'e}phane Boucheron, G{\'a}bor Lugosi, and Pascal Massart,
  \emph{Concentration {{Inequalities}}: {{A Nonasymptotic Theory}} of
  {{Independence}}}, {OUP Oxford}, February 2013.

\bibitem[Bol80]{Bol80}
B{\'e}la Bollob{\'a}s, \emph{A probabilistic proof of an asymptotic formula for
  the number of labelled regular graphs}, European Journal of Combinatorics
  \textbf{1} (1980), no.~4, 311--316.

\bibitem[CGS15]{chatterjee_matrix_2015}
Sabyasachi Chatterjee, Adityanand Guntuboyina, and Bodhisattva Sen, \emph{On
  matrix estimation under monotonicity constraints}, arXiv preprint
  arXiv:1506.03430 (2015).

\bibitem[Chu97]{Chu97}
Fan~RK Chung, \emph{Spectral graph theory}, vol.~92, {American Mathematical
  Soc.}, 1997.

\bibitem[DHL14]{DalHebLed14}
Arnak~S. Dalalyan, Mohamed Hebiri, and Johannes Lederer, \emph{On the
  prediction performance of the lasso}, to appear in Bernoulli, arXiv
  1402.1700, February 2014.

\bibitem[DJ94]{DonJoh94}
David~L. Donoho and Jain~M. Johnstone, \emph{Ideal spatial adaptation by
  wavelet shrinkage}, Biometrika \textbf{81} (1994), no.~3, 425--455.

\bibitem[DJ95]{DonJoh95}
David~L. Donoho and Iain~M. Johnstone, \emph{Adapting to unknown smoothness via
  wavelet shrinkage}, J. Amer. Statist. Assoc. \textbf{90} (1995), no.~432,
  1200--1224. \MR{MR1379464 (96k:62093)}

\bibitem[Fri04]{Fri04}
J~Friedman, \emph{A proof of {Alon}’s second eigenvalue conjecture and
  related problems}, Mem. Amer. Math. Soc \textbf{195} (2004), no.~910.

\bibitem[Gir14]{Gir14}
Christophe Giraud, \emph{Introduction to high-dimensional statistics}, {CRC
  Press}, 2014.

\bibitem[HLL12]{HarLev12}
Za\i{}d Harchaoui and C{\'e}line L{\'e}vy-Leduc, \emph{Multiple change-point
  estimation with a total variation penalty}, Journal of the American
  Statistical Association (2012).

\bibitem[KOV14]{KolOstVon14}
Theodore Kolokolnikov, Braxton Osting, and James {Von Brecht}, \emph{Algebraic
  connectivity of {{Erd{\"o}s-R{\'e}nyi}} graphs near the connectivity
  threshold}, Manuscript in preparation (2014).

\bibitem[KT93]{KorTsy93}
A.~P. Korostelev and A.~B. Tsybakov, \emph{Minimax {{Theory}} of {{Image
  Reconstruction}}}, Lecture Notes in Statistics, vol.~82, {Springer New York},
  New York, NY, 1993.

\bibitem[Mv97]{mammen_locally_1997}
Enno Mammen and Sara {van de Geer}, \emph{Locally adaptive regression splines},
  The Annals of Statistics \textbf{25} (1997), no.~1, 387--413.

\bibitem[NW13a]{NeeWar13a}
Deanna Needell and Rachel Ward, \emph{{\mockalph{a}S}table image reconstruction
  using total variation minimization}, SIAM Journal on Imaging Sciences
  \textbf{6} (2013), no.~2, 1035--1058.

\bibitem[NW13b]{NeeWar13b}
\bysame, \emph{{\mockalph{b}N}ear-optimal compressed sensing guarantees for
  total variation minimization}, IEEE Transactions on Image Processing
  \textbf{22} (2013), no.~10, 3941--3949.

\bibitem[OV15]{OllVia15}
Edouard Ollier and Vivian Viallon, \emph{Regression modeling on stratified
  data: automatic and covariate-specific selection of the reference stratum
  with simple ${L}_1$-norm penalties}, arXiv:1508.05476 {[}math, stat] (2015).

\bibitem[Pun10]{Pun10}
Golan Pundak, \emph{Random {{Regular Generator}}}, {MATLAB Central File
  Exchange} (2010).

\bibitem[QJ12]{QiaJia12}
Junyang Qian and Jinzhu Jia, \emph{On pattern recovery of the fused {{Lasso}}},
  arXiv:1211.5194 (2012).

\bibitem[Rin09]{Rin09}
A.~Rinaldo, \emph{Properties and refinements of the fused lasso}, The Annals of
  Statistics \textbf{37} (2009), no.~5B, 2922--2952.

\bibitem[Roc70]{Roc70}
Ralph~Tyrell Rockafellar, \emph{Convex analysis}, {Princeton university press},
  1970.

\bibitem[ROF92]{RudOshFat92}
Leonid~I. Rudin, Stanley Osher, and Emad Fatemi, \emph{Nonlinear total
  variation based noise removal algorithms}, Physica D: Nonlinear Phenomena
  \textbf{60} (1992), no.~1, 259--268.

\bibitem[She10]{She10}
Yiyuan She, \emph{Sparse regression with exact clustering}, Electronic Journal
  of Statistics \textbf{4} (2010), 1055--1096.

\bibitem[SSR12]{ShaSinRin12}
James Sharpnack, Aarti Singh, and Alessandro Rinaldo, \emph{Sparsistency of the
  edge lasso over graphs}, Proceedings of the Fifteenth International
  Conference on Artificial Intelligence and Statistics (AISTATS-12) (Neil~D.
  Lawrence and Mark~A. Girolami, eds.), vol.~22, 2012, pp.~1028--1036.

\bibitem[Str07]{Str07}
Gilbert Strang, \emph{Computational science and engineering}, vol.~1,
  {Wellesley-Cambridge Press Wellesley}, 2007.

\bibitem[TSR{\etalchar{+}}05]{TibSauRos05}
Robert Tibshirani, Michael Saunders, Saharon Rosset, Ji~Zhu, and Keith Knight,
  \emph{Sparsity and smoothness via the fused lasso}, Journal of the Royal
  Statistical Society: Series B (Statistical Methodology) \textbf{67} (2005),
  no.~1, 91--108.

\bibitem[VLLHP16]{ViaLamHoe16}
Vivian Viallon, Sophie Lambert-Lacroix, H{\"o}lger Hoefling, and Franck Picard,
  \emph{On the robustness of the generalized fused lasso to prior
  specifications}, Statistics and Computing \textbf{26} (2016), no.~1-2,
  285--301.

\bibitem[WNC05]{willett_faster_2005}
Rebecca Willett, Robert Nowak, and Rui~M. Castro, \emph{Faster rates in
  regression via active learning}, Advances in {{Neural Information Processing
  Systems}}, 2005, pp.~179--186.

\bibitem[WSST15]{WanShaSmo15}
Yu-Xiang Wang, James Sharpnack, Alex Smola, and Ryan Tibshirani, \emph{Trend
  {{Filtering}} on {{Graphs}}}, Proceedings of the {{Eighteenth International
  Conference}} on {{Artificial Intelligence}} and {{Statistics}}, 2015,
  pp.~1042--1050.

\bibitem[XKWG14]{XinKawWan14}
Bo~Xin, Yoshinobu Kawahara, Yizhou Wang, and Wen Gao, \emph{Efficient
  {{Generalized Fused Lasso}} and {{Its Application}} to the {{Diagnosis}} of
  {{Alzheimer}}'s {{Disease}}}, Twenty-{{Eighth AAAI Conference}} on
  {{Artificial Intelligence}}, 2014, pp.~2163--2169.

\end{thebibliography}

\appendix

\section{Numerical experiments} % {{{3
\label{sub:numerical_experiments}

\begin{figure}
\captionsetup[subfigure]{justification=centering}
	\centering
	\begin{subfigure}[t]{0.48\textwidth}
		\includegraphics[width=\textwidth]{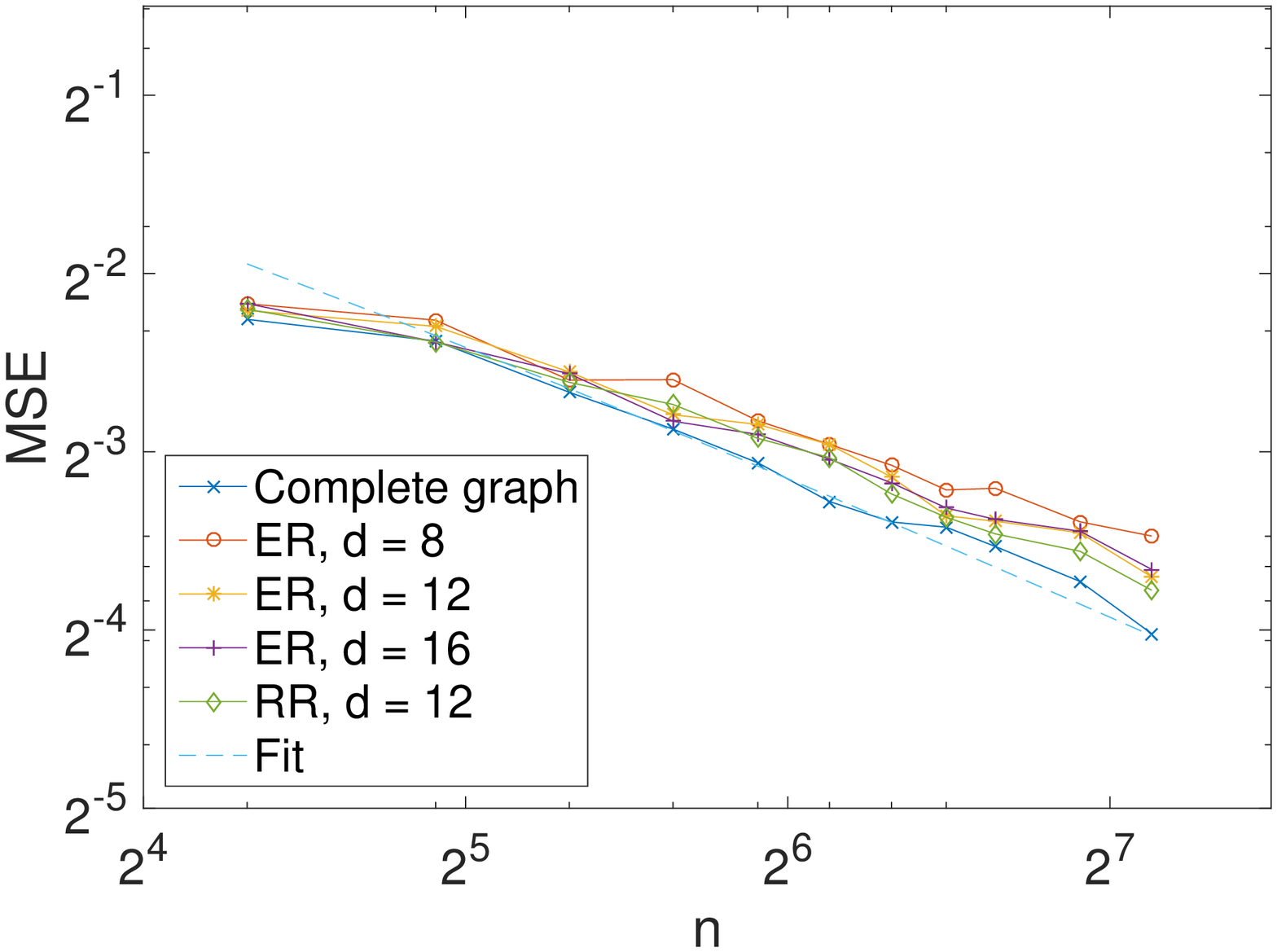}
		\centering \caption{Oracle choice of \( \lambda \)}
		\label{fig:mse-oracle}
	\end{subfigure}
	~
	\begin{subfigure}[t]{0.48\textwidth}
		\includegraphics[width=\textwidth]{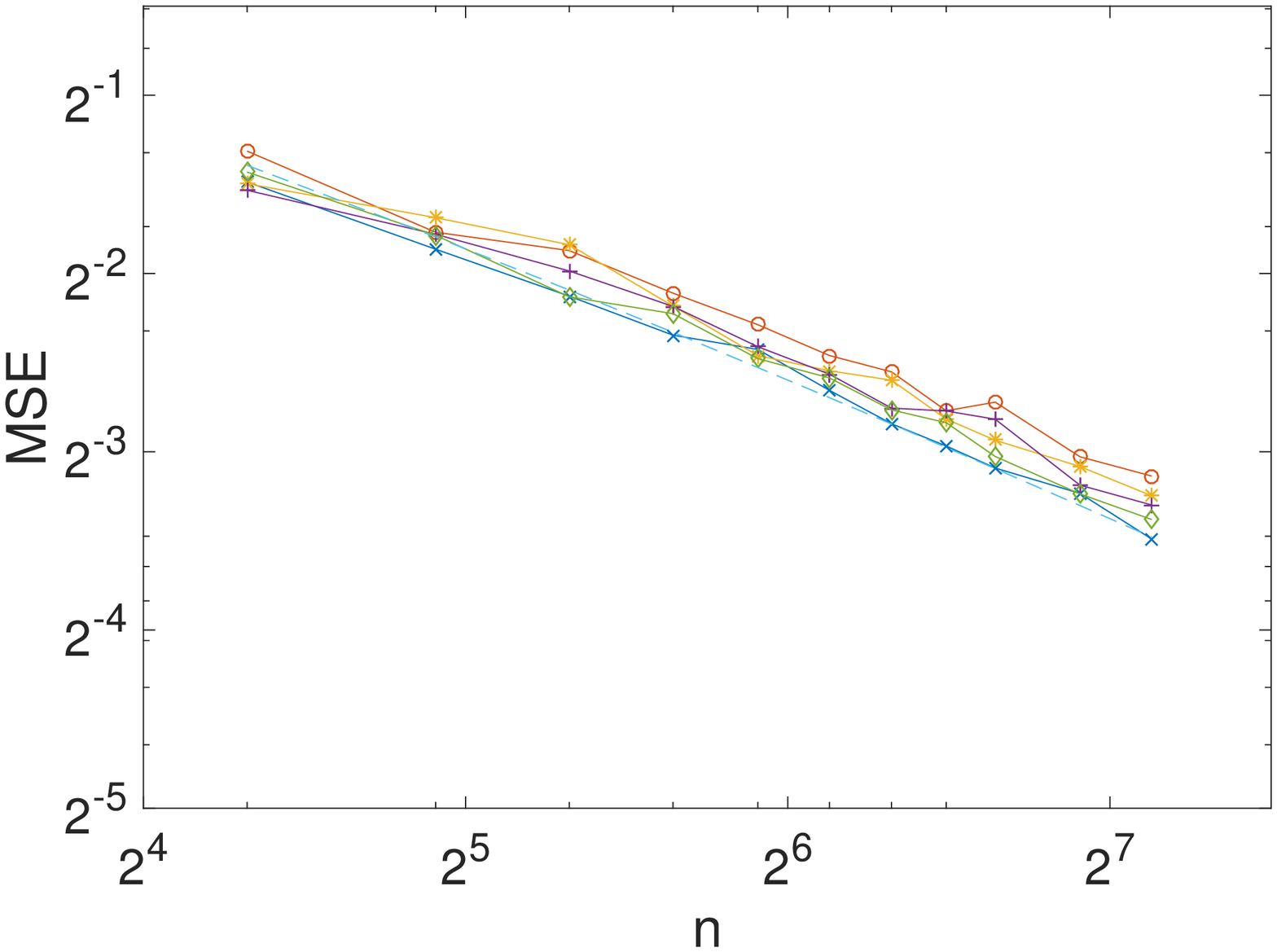}
		\caption{Fixed choice of \( \lambda = \sigma \rho \sqrt{\log(n)}\)}
		\label{fig:mse-fixed}
	\end{subfigure}
	\caption{MSE for the Island model with $k=l=3$ for different choices of the graph and different choices of the regularization parameter $\lambda$. The dotted line the best fit of form $C\log(n)/n$ to the complete graph case. }
	\label{fig:mse}
\end{figure}

\begin{figure}
	\centering
		\includegraphics[width=0.47\textwidth]{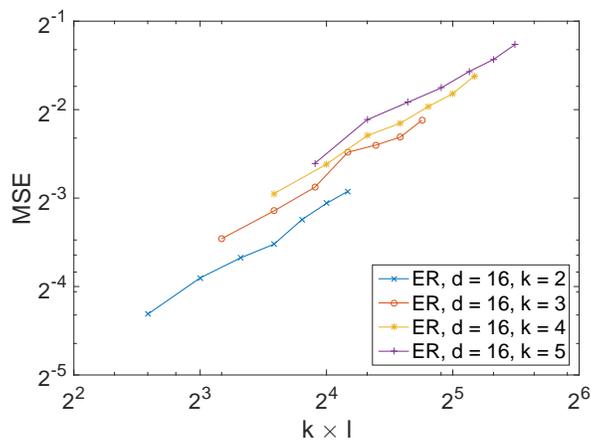}
		\caption{MSE for the Island model for different choices of $k\cdot l$ ($n=100, \lambda= \sigma \rho \sqrt{\log(n)}$).} 
		\label{fig:islands-fixed}
\end{figure}

In order to illustrate our findings in Subsections \ref{sub:complete_graph} and \ref{sub:random-graphs}, we used the TV denoiser implementation from \cite{XinKawWan14}.

\smallskip

\noindent {\it The Island model.} Consider a partition of $[n]$ into  \( k  \) blocks $B_1, \ldots, B_{k}$ of size \(|B_j|= l, l \in [k] \) and a block $B_0$ of size $|B_0|=n-kl$. We focus on cases where $n \gg kl$ and we call block $B_0$, the \emph{background component} and the blocks $B_j, j \in k$ are called \emph{islands}. The unknown parameter $\theta^*$ has coordinates $\theta^*_i=50 + 10j, i \in B_j, j \in k$, and $\theta^*_l= 50, i \in B_0$. 

\smallskip

\noindent {\it Graphs. }We consider three types of graphs to determine our penalty structure: the complete graph, the Erd\H{o}s-R\'enyi random graph with expected degree $d$ and the random $d$-regular graph\footnote{To generate instances of the random regular graph, we employed the code from \cite{Pun10}, which implements the pairing algorithm by Bollob\'as, \cite{Bol80}.} for different values of $d$. Note that in the case of the random graphs, we refer to a realization from a given distribution as ``the" random graph.

\smallskip

\noindent {\it Choice of $\lambda$. } We consider two choices for the regularization parameter \( \lambda \): the fixed choice, denote by $\lambda_{\textrm{th}}$ dictated by our theoretical results and an oracle choice $\lambda_{\textrm{or}}$ on a geometric grid, obtained by $\lambda_{\textrm{or}}=10\lambda_{\textrm{th}} \beta^{j^*}$ where $j^*$ is the smallest $j \ge 1$ such that $\|\hat \theta(10\lambda_{\textrm{th}}\beta^{j^*+i})-\theta^*\|_2\ge \|\hat \theta(10\lambda_{\textrm{th}}\beta^{j^*})-\theta^*\|_2$ for \( i = 1, 2, 3 \), and $\hat\theta(\lambda)$ is the solution to~\eqref{eq:y} and \( \beta = 0.85 \).

\smallskip

Throughout the simulations, we choose $\sigma=0.5$. The plotted results are averaged over 50 realizations of the noise and, in the case of a random graph, over realizations of said random graph.

In Figure~\ref{fig:mse}, we consider the Island model with $k=3$ islands, each of size $l=3$. We plot (on a log-log scale) the mean squared error of the TV denoiser as  a function of $n$ for both the oracle choice and the theoretical choice of $\lambda$ for different graph models: the complete graph, the Erd\H{o}s-R\'enyi random graphs with expected degree $d$ for $d=2,12, 16$ and the random $12$-regular graph. The dotted line indicates  the best fit of the form $C\log(n)/n$  that our theoretical analysis predicts in the complete graph case. 
In all cases, we can see that the mean squared error essentially scales as $C(\log n)/n$ as predicted by our theory. Moreover, all graphs show similar performance, though the sparse ones lead to better computational performance.

The purpose of Figure \ref{fig:islands-fixed} is to illustrate that the scaling $kl/n$ for the model with islands obtained in subsection~\ref{sub:random-graphs} is indeed the correct one. In this set of simulations we use the  Erd\H{o}s-R\'enyi graph with expected degree $d=16$ and plot the mean squared error for different values of the pair $(k,l)$. Specifically, we choose $(k,l)\in [2:5]\times[3:9]$ and indeed observe a linear dependence on the product~$kl$.

\section{Proofs}

\subsection{Proof of the main theorem: a sharp oracle inequality for TV denoising}
\label{proof:main}

In this subsection, we prove Theorem~\ref{thm:Lasso-rates} that we recall for convenience

 \begin{theorem*}[Sharp oracle inequality for TV denoising]
Fix  \( \delta \in (0,1) \), $T \subset [m]$ and let \( D \) being the incidence matrix of a connected graph \( G \). Define the regularization parameter
	\begin{equation}
		\label{eq:j}
		\lambda := \frac{1}{n} \sigma \rho \sqrt{2 \log\big(\frac{em}{\delta}\big)}.
	\end{equation}
With this choice of $\lambda$, the TV denoiser $\hat \theta$ defined in \eqref{eq:y}  satisfies
	\begin{align}
		\label{eq:q}
		 \frac{1}{n} \| \hat\theta - \theta^\ast \|^2
		\leq {} & \inf_{\bar{\theta} \in \R^n} \left\{ \frac{1}{n} \| \bar{\theta} - \theta^\ast \|^2 + 4 \lambda \| (D \bar{\theta})_{T^c} \|_1 \right\}
		+ \frac{8 \sigma^2}{n} \left( \frac{|T| \rho^2}{\kappa_T^{2}} \log \big(\frac{em}{\delta}\big) + \log\big(\frac{e}{\delta}\big) \right),
	\end{align}
	on the estimation error with probability at least \( 1 - 2 \delta \). 
\end{theorem*}

Our proof is based on the sharp oracle inequality for the Lasso in \cite[Theorem 4.1, Corollary 4.3]{Gir14} and slightly stronger statements that appear in \cite[Theorems~3 and~4]{DalHebLed14}.
 
\begin{proof}
	We start by considering the first order optimality conditions of the convex problem \eqref{eq:y}.
 	By the chain rule for the subdifferential, \cite[Theorem 23.9]{Roc70}, the subdifferential of the \( \ell_1 \) term is
	\begin{equation}
		\label{eq:am}
		\partial \| D \theta \|_1 = D^\top \sign (D \theta),
	\end{equation}
	where
$$
\sign(x)_i = \left\{
\begin{array}{rl}
1 & \text{if}\ x_i > 0,\\
 \left[-1 , 1\right] & \text{if}\ x_i = 0,\\
-1 & \text{if}\ x_i < 0\,.
\end{array}
\right.
$$
	Therefore, for any $\bar \theta \in \R^n$,  $z \in \sign(D \hat\theta)$ we get
%	\begin{align}
%		\label{eq:o}
%		\frac{1}{n} (y - \hat\theta) = \lambda D^\top z, \quad z \in \sign(D \hat\theta).
%	\end{align}
%	That means when we multiply \( (y - \hat\theta)/n \) by any vector \( \bar{\theta} \), we get
	\begin{equation}
		\label{eq:ao}
		\frac{1}{n} \bar{\theta}^\top (y - \hat\theta) = \lambda \bar{\theta}^\top D^\top z = \lambda (D \bar{\theta})^\top z.
	\end{equation}
	It yields
%	Since \( z \) has entries in \( [-1, 1] \) that match the signs of \( D \hat\theta \), we get
	\begin{align*}
%		\label{eq:p}
		\frac{1}{n} \hat\theta^\top(y - \hat\theta) = {} \lambda \| D \hat\theta \|_1\quad \text{and} \quad  
		\frac{1}{n} \bar{\theta}^\top (y - \hat\theta) \leq {}  \lambda \| D \bar{\theta} \|_1. %\label{eq:ap}
	\end{align*}
In turn, subtracting the above two, we get
	\begin{align}
		\label{eq:a}
		\frac{1}{n} (\bar{\theta} - \hat\theta)^\top (\theta^\ast - \hat\theta)
		\leq \frac{1}{n} \varepsilon^\top ( \hat\theta - \bar{\theta} ) + \lambda \| D\bar{\theta} \|_1 - \lambda \| D\hat\theta \|_1.
	\end{align}
Next, using  polarization, we can rewrite the above display as
	\begin{equation}
		\label{eq:ar}
		\frac{1}{n}( \| \bar{\theta} - \hat\theta \|^2 + \| \theta^\ast - \hat\theta \|^2)
		\leq \frac{1}{n} \| \bar{\theta} - \theta^\ast \|^2 + \frac{2}{n} \varepsilon^\top ( \hat\theta - \bar{\theta} ) + 2\lambda \| D\bar{\theta} \|_1 - 2\lambda \| D\hat\theta \|_1.
	\end{equation}

We first control the error term \( \varepsilon^\top ( \hat\theta - \bar{\theta} ) \) as follows.
	Let \( \Pi \) denote the projection matrix onto \( \operatorname{ker}(D) \) and remember that \( D^\dag D = (I - \Pi) \), the projection on \( \operatorname{ker}(D)^\perp \).
	Since \( \operatorname{ker}(D) = \operatorname{ker}(D^\top D) \), the kernel of the graph Laplacian, and \( G \) is connected, we have \( \operatorname{ker}(D) = \operatorname{span}(\bone_n) \)  \cite{Chu97}; in particular, \( \operatorname{dim} \operatorname{ker}(D) = 1 \).
It yields
	\begin{align}
		\varepsilon^\top(\hat\theta - \bar{\theta})
		&=(\Pi \varepsilon)^\top ( \hat\theta - \bar{\theta}) + ((I - \Pi)\varepsilon)^\top ( \hat\theta - \bar{\theta})\\
		&=(\Pi \varepsilon)^\top ( \hat\theta - \bar{\theta}) + \varepsilon^\top D^\dag D ( \hat\theta - \bar{\theta})\\
		% &=(\Pi \varepsilon)^\top ( \hat\theta - \bar{\theta}) + ((D^\dag)^\top \varepsilon)^\top D ( \hat\theta - \bar{\theta}) \\
		&\leq\| \Pi \varepsilon \| \|  \widehat{\theta }- \bar{\theta} \| + \|(D^\dag)^\top \varepsilon\|_\infty \| D( \hat\theta - \bar{\theta}) \|_1\,,\label{eq:b}
	\end{align}
where in~\eqref{eq:b}, we use H\"older's inequality.

To bound the right-hand side in~\eqref{eq:b}, we first use the maximal inequality for Gaussian random variables \cite[Corollary 2.6]{BouLugMas13}:
It yields that the following two inequalities hold simultaneously on an event of probability $1-2\delta$,
	\begin{align}
		\| (D^\dag)^\top \varepsilon \|_\infty
		\leq {}  \sigma \rho \sqrt{2 \log (em/\delta)} = \lambda n\,, \qquad
		\| \Pi \varepsilon \|_2
		\leq {}  2\sigma\sqrt{2 \log(e/\delta)}\,. \label{eq:as}
	\end{align}
Next, note that by the triangle inequality we have
%\begin{align}
%	\label{eq:c}
%	\| D( \hat\theta - \bar{\theta}) \|_1 &\le  \| (D( \hat\theta - \bar{\theta}))_T \|_1 + \| (D \bar{\theta})_{T^c} \|_1 + \| (D \hat\theta)_{T^c} \|_1, \\
%	\| D \bar{\theta} \|_1 - \| D \hat\theta \|_1   &\le \| (D( \hat\theta - \bar{\theta})_T \|_1 + \| (D \bar{\theta})_T \|_1 - \| (D \hat\theta)_{T^c} \|_1,
%\end{align}
%which combined yields
\begin{equation}
	\label{eq:r}
	\| D( \hat\theta - \bar{\theta}) \|_1 + \| D \bar{\theta} \|_1 - \| D \hat\theta \|_1
	\leq 2 \| (D (\hat\theta - \bar{\theta}))_T \|_1 + 2\| (D \bar{\theta})_{T^c} \|_1.
\end{equation}
Moreover, $\| D ( \hat\theta - \bar{\theta} )_T \|_1 \leq \kappa_T^{-1} \sqrt{|T|} \| \hat\theta - \bar{\theta} \|$. Together with~\eqref{eq:ar}--\eqref{eq:r}, it yields
\begin{align*}
	\frac{1}{n}( \| \bar{\theta} - \hat\theta \|^2 + \| \theta^\ast - \hat\theta \|^2)
	\le {} & \frac{1}{n} \| \bar{\theta} - \theta^\ast \|^2 + 4 \lambda \| (D\bar{\theta})_{T^c} \|_1\\
	&+ \frac{4}{n} \| \hat\theta - \bar{\theta} \| \left(\sigma \sqrt{2 \log \big(e/\delta\big)} +  n \frac{\lambda}{\kappa_T} \sqrt{|T|} \right).
\end{align*}
To conclude the proof, we apply Young's inequality to  produce \( \frac{1}{n} \| \hat\theta - \bar{\theta} \|^2 \) which cancels out.
\end{proof}

\subsection{Control of the inverse scaling factor for the 2D grid}
\label{proof:2d}
In this subsection, we prove Proposition~\ref{prp:tv-cols-2d} that we recall here for convenience.

\begin{proposition*}
The incidence matrix \( D_2 \) of the 2D grid on $n$ vertices has inverse scaling factor $\rho\lesssim\sqrt{\log n}$.
\end{proposition*}
\begin{proof}

Note first that $S = D_2^\dag = (D_2^\top D_2)^\dag D_2^\top$. Moreover,  the matrix \( D_2^\top D_2 \) can be expressed in terms of $D_1^\top D_1$ as
	\begin{equation}
		\label{eq:g}
		D_2^\top D_2 =
		\begin{bmatrix}
			D_1^\top \otimes I & I \otimes D_1^\top
		\end{bmatrix}
		\begin{bmatrix}
			D_1 \otimes I\\
			I \otimes D_1
		\end{bmatrix}
		= D_1^\top D_1 \otimes I + I \otimes D_1^\top D_1.
	\end{equation}
It follows from~\cite[Chapter 1.5]{Str07} that the unnormalized Laplacian $D_1^\top D_1$ of the path graph admits the following spectral decomposition
\begin{equation}
\label{EQ:toeplitzmat}
D_1^\top D_1 = 
		\begin{bmatrix}
			1 & -1 & 0 & 0 & \dots & 0\\
			-1 & 2 & -1 & 0 & \dots & 0\\
			0 & -1 & 2 & -1 & \dots & 0\\
			\vdots & & \ddots & & & \vdots \\
			\vdots & & & -1 & 2 & -1\\
			0 & & \dots & 0 & -1 & 1
		\end{bmatrix}
		= V_1 \Lambda_1 V_1^\top
\end{equation}
where $\Lambda_1=\operatorname{diag}(\lambda_0, \dots, \lambda_{N-1})$, with
	\begin{equation}
		\label{eq:k}
		\lambda_k = 2 - 2 \cos \frac{k \pi}{N}\,, \quad k \in \llbracket0, N \llbracket\,,
	\end{equation}
	and $V_1=[v_0, \ldots, v_{N-1}]$  is the discrete Fourier transform {\sc Dct-2} on $\R^N$ so that each eigenvector $v_k \in \R^N$ has coordinates
	\begin{align}
		(v_0)_j = {} & \frac{1}{N}, \quad j \in \llbracket 0, N \llbracket\\
		(v_k)_j = {} & \sqrt{\frac{2}{ N }} \cos \left(\frac{(j + 1/2)k\pi}{ N }\right)\,, \quad j \in \llbracket0, N \llbracket, \, k \in \llbracket 1, N \llbracket\,.
	\end{align}
Therefore, $D_2^\top D_2 =V_2\Lambda_2V_2^\top$\,, where   \( \Lambda_2 = \Lambda_1 \otimes I + I \otimes \Lambda_1 \) and \( V_2 = V_1 \otimes V_1 \)\,.

As a result,  $S$ has $2N (N-1)$ columns and can be written as
$$
S=D_2^\dag=V_2\Lambda_2^{\dag} V_2^\top 
\begin{bmatrix}
	D_1^\top \otimes I & I \otimes D_1^\top 
\end{bmatrix}
= [(s_{i, j}^{(1)})_{\subalign{&i \in [N-1]\\&j \in [N]}}, (s_{i, j}^{(2)})_{\subalign{&i \in [N]\\&j \in [N-1]}}].
$$
Write $D_1^\top = [d_1, \dots, d_{N-1}]$ and note that the columns of $S$ have norm given for  $\diamond \in \{1,2\}$ by
%$$
%	\begin{equation}
%		\label{eq:m}
%		D_1^\top = [d_1, \dots, d_{N-1}]\,, \quad S = [(s_{i, j}^{(1)})_{i \in [N-1],j \in [N]}, (s_{i, j}^{(2)})_{i \in [N], j \in [N-1]}]\,,
%	\end{equation}
%where the superscipt for the columns of \( S \) corresponds to the two components of \( D_2 \), operating on the first and second components of a grid element, respectively.
%	Because the columns of \( V_2 \) form an orthonormal basis, and by excluding the \( 0 \) eigenvalue in the sum by the definition of the pseudo-inverse, we can write
	\begin{align}
		\label{eq:l}
		\| s_{i, j}^{(\diamond)} \|_2^2
		&= \sum_{\substack{k,l = 0\\(k,l) \neq (0,0)}}^{N-1} \frac{1}{(\lambda_k + \lambda_l)^2} \langle v_{k} \otimes v_l , d_i \otimes e_j \rangle^2\\
		&= \sum_{\substack{k,l = 0\\(k,l) \neq (0,0)}}^{N-1} \frac{1}{(4 - 2\cos \frac{k \pi}{ N } - 2 \cos \frac{l \pi}{ N })^2} \langle v_k , d_i \rangle^2 \langle v_l , e_j \rangle^2\,,
	\end{align}
%	and similarly for \( s_{i, j, 2} \) by swapping \( k \) and \( l \), so we can without loss of generality bound this expression.
%	
where $e_0,\ldots, e_{N-1}$ are the vectors of the canonical basis of $\R^N$. Next, note that
	\begin{equation}
		\label{eq:cd}
		\langle v_l , d_i \rangle^2 = \frac{2}{N} \left( \cos \frac{l \pi(i + 3/2) }{N} - \cos \frac{l \pi(i + 1/2)}{N}  \right)^2
		\le  \frac{2l^2 \pi^2}{N^3}\,, %\sin \frac{l \pi}{N} \xi
	\end{equation}
because $x \mapsto \cos x$ is 1-Lipschitz. 
%	It 
%	\begin{equation}
%		\label{eq:ce}
%		\langle v_l , d_i \rangle^2 \lesssim \frac{l^2}{N^3}. 
%	\end{equation}
Moreover, we have that  $\langle v_k , e_j \rangle^2 \le 2/N$.
%since \( D_2 \) is an incidence matrix, every row has only two non-zero entries, so we can exploit that the eigenvectors are delocalized in the sense that \( \left\| v_{k} \right\|_\infty \leq \sqrt{2/N} \) to see that
%	\begin{equation}
%		\label{eq:cf}
%		\langle v_k , e_j \rangle^2 \lesssim \frac{1}{N}.
%	\end{equation}
%

	It remains to bound the sum. To that end, observe that  \( 2 - 2 \cos x  \geq x^2/2 \) for any \( x \in [0,1/2] \) and  \( 2 - 2 \cos x \geq  0.1 \), for \( x \in [1/2, \pi] \).
	Hence, we can split the sum into to parts to get
	\begin{align}
		\label{eq:n}
		\left\|s_{i, j}^{(\diamond)}\right\|_2^2 \leq {} & \frac{4\pi^2}{N^4} \smashoperator{\sum_{\substack{k,l = 0\\(k,l) \neq (0,0)}}^{ N -1}} \frac{l^2}{(4 - 2 \cos \frac{k \pi}{ N } - 2 \cos \frac{l \pi}{ N })^2}\\
				\leq {} & \frac{4\pi^2}{N^4} \smashoperator{\sum_{\substack{k,l = 0\\(k,l) \neq (0,0)}}^{ N-1}} \frac{l^2}{(4 - 2 \cos \frac{k \pi}{ N } - 2 \cos \frac{l \pi}{ N })^2}
				\big[\1_{\{\frac{2\pi}{N}(k\vee l) \le 1\}}+ \1_{(\frac{2\pi}{N}\{k\vee l) > 1\}}\big]\\
%		{} & + \frac{4\pi^2}{N^4} \sum_{\substack{k,l = 0\\(k,l) \neq (0,0)}}^{ N-1} \frac{l^2}{(4 - 2 \cos \frac{k \pi}{ N } - 2 \cos \frac{l \pi}{ N })^2}\big(\1(k \pi/ N  \leq 1/2, \, \pi/ N  \leq 1/2)+ \1(\{k \pi/ N  > 1/2 \text{ or } l \pi/ N  > 1/2)\big)\\
		\leq {} & 16 \smashoperator{\sum_{\substack{k,l = 0\\(k,l) \neq (0,0)}}^{ N-1}} \frac{l^2}{(k^2 + l^2)^2} + \frac{400\pi^2}{N^3} \sum_{k = 0}^{ N-1} k^2
		\lesssim \sum_{\substack{k,l = 1}}^{ N-1} \frac{l^2}{(k^2 + l^2)^2}  + 1\,.
	\end{align}
%	where in the last line, we used that \( \sum_{k = 1}^{\infty} k^{-2} = \pi^2/6 \) and \( \sum_{k = 1}^{N} l^2 \lesssim N^3 \).
Using a comparison between series and integral, noting that \( x \to x^2/(k^2 + x^2)^2 \) is increasing on \( [0,k^2] \) and decreasing on \( [k^2, \infty) \), it is immediate that
	\begin{align}
		\label{eq:ca}
	\sum_{\substack{k,l = 1}}^{ N-1} \frac{l^2}{(k^2 + l^2)^2}\le \sum_{k = 1}^{ N-1} \frac{1}{k} \int_0^\infty \frac{x^2}{(1 + x^2)^2} \diff x + \sum_{k=1}^{N-1} \frac{1}{4 k^2}
		\lesssim  \sum_{k = 1}^{ N } \frac{1}{k} + 1
		\lesssim \log N.
	\end{align}
To conclude the proof, observe that $n=N^2$.
\end{proof}

\subsection{Control of the inverse scaling factor for high-dimensional grids}
\label{proof:3d}
In this subsection, we prove Proposition~\ref{prp:tv-cols-dd} that we recall here for convenience.

\begin{proposition*}
	For the incidence matrix of the regular grid on \( N^d \) nodes in \( d \) dimensions, \( \rho \leq C(d) \), for some \( C(d) > 0 \).
\end{proposition*}

\begin{proof}
	Similarly to the proof of Proposition \ref{prp:tv-cols-2d}, \( (D_d^\top D_d)^\dag \) admits an eigendecomposition of the form \( \Lambda_d = \Lambda_1 \otimes I \otimes \dots \otimes I + \dots + I \otimes I \otimes \dots \otimes \Lambda_1 \), \( V_d = V_1^{\otimes d} \).
	Keeping the same notation as in the preceding proof,
	\begin{equation}
		\label{eq:bi}
		S = D_d^\dag = [(s_{\bm{i}}^{(j)})_{ i_j \in [N-1], \, i_k \in [N],} \text{ for } k \neq j, \, j \in [d]],
	\end{equation}
	we have
	\begin{align}
		\label{eq:bh}
		\leadeq{\| s_{ \bm{i}}^{(1)} \|_2^2
		= \sum_{\substack{k_l = 0, \, \bm{k} \neq 0\\l = 1,\dots,d}}^{N-1} \left(\sum_{j=1}^{d} \lambda_{k_j}\right)^{-2} \langle v_{k_1} , d_{i_1} \rangle^2 \prod_{j = 2}^{d} \langle v_{k_j} , e_{i_j} \rangle^2}\\
		= {} & \sum_{\substack{k_l = 0, \, \bm{k} \neq 0\\l = 1,\dots,d}}^{N-1} \left(\sum_{j=1}^{d} \left(2 - 2 \cos \frac{k_j \pi}{N}\right)\right)^{-2} \langle v_{k_1} , d_{i_1} \rangle^2 \prod_{j = 2}^{d} \langle v_{k_j}, e_{i_j} \rangle^2
	\end{align}
	and by symmetry, this case is enough to deduce the claim for an arbitrary \( s_{ \bm{i}, j} \), \( j \in [d] \).
	Observing again that 
	\begin{equation}
		\label{eq:bk}
		\|v_{k_j} \|_\infty \leq \sqrt{2/N}, \quad \langle v_{k_j} , e_{i_j} \rangle^2 \leq 2 / N,
	\end{equation}
	and
	\begin{equation}
		\label{eq:cg}
		\langle v_{k_1} , d_{i_1} \rangle^2 \leq \frac{2 k_1^2}{N^3},
	\end{equation}
	it remains to bound the sum above.

	For this, use the same bounds on the cosine function to split it up into a part bounded by a constant and one that behaves like a square:
  \begin{align}
		\| s_{ \bm{i}}^{(1)} \|^2 \leq {} & \frac{2^d}{N^{d+2}}\sum_{\substack{ k_l = 0, \, \bm{k} \neq 0\\l=1,\dots,d}}^{N-1} k_1^2 \left( 2d - 2 \sum_{j=1}^{d} \cos \frac{k_j \pi}{N} \right)^{-2}\\
		% \leq {} & \frac{2^d}{N^{d+2}} \sum_{{\substack{k_l = 0\\\bm{k} \neq 0}}}^{ N-1} k_1^2 \left( 2d - 2 \sum_{j=1}^{d} \cos \frac{k_j \pi}{N} \right)^{-2}
		% \left(\1_{\{ \forall j:\; k_j \pi/N \leq 1/2\}} + \1_{\{\exists j:\; k_j \pi/N > 1/2\}} \right) \\
		\lesssim {} & \frac{2^d}{N^{d-2}} \sum_{{\substack{k_l = 0\\ \bm{k} \neq 0}}}^{ N-1 } k_1^2 \left( \sum_{j=1}^{d} k_j^2 \right)^{-2} + 1 \label{eq:bj}
	\end{align}
	We again want to exclude all indices having a zero element.
	This amounts to finding a bound of the order \( o(N^{d+2}) \) for the same sum in one dimension less than we are considering here, times \( d \) for each coordinate that can be zero.
	In order to achieve this, we argue by induction: in \( d=3 \) dimensions, the corresponding summation runs over two indices and has been shown to be of order \( O(\log n) = o(N) \) in the proof of Proposition \ref{prp:tv-cols-2d}, so the base case is valid.
	The following analysis will show that the whole sum is \( O(N^{d+2}) \) for \( d \geq 3 \), which is the induction step.
	This means we can assume
	\begin{align}
		\label{eq:cb}
		\| s_{ \bm{i}}^{(1)} \|^2 \leq {} & \frac{2^d}{N^{d-2}} \sum_{{\substack{k_l = 1\\l = 1,\dots,d}}}^{ N-1}  k_1^2 \Big(\sum_{j=1}^{d} k_j^2\Big)^{\mathrlap{-2}} + o(d)
		\lesssim {}  \frac{2^d}{N^{d-2}} \sum_{{\substack{k_l = 1\\l = 1,\dots,d}}}^{ N-1} k_1^2 \Big( \sum_{j=1}^{d} k_j^2 \Big)^{-2} + 1.
	\end{align}
Next, observe that \( \int_{0}^{\infty} x^2(1+x^2)^{-2} \diff x \lesssim 1 \). It yields
\begin{align}
	\leadeq{\sum_{{\substack{k_l = 1\\l=1, \dots, d}}}^{ N-1} k_1^2 \Big( \sum_{j=1}^{d} k_j^2 \Big)^{-2}}\\
	&\leq  \frac{2^d}{N^{d-2}} \sum_{{\substack{k_l = 1\\l=2, \dots, d}}}^{ N-1} \smashoperator{\int_{0}^{ \infty}} x^2 \Big( x^2 + \sum_{j=2}^{d} k_j^2 \Big)^{-2} \diff x + \frac{2^d}{N^{d-2}} \smashoperator{\sum_{\substack{k_j = 1\\j = 2, \dots, d}}^{N-1}} \Big( \sum_{j=2}^{d} k_j^2 \Big)^{-1}\\
	&= \frac{2^d}{N^{d-2}} \smashoperator{\sum_{{\substack{k_j = 1\\j=2, \dots, d}}}^{ N-1}} \Big(\sum_{j=2}^{d} k_j^2\Big)^{-1/2} \smashoperator{\int_{0}^{ \infty}} \frac{y^2}{ (y^2 + 1)^2 } \diff x + \frac{2^d}{N^{d-2}} \smashoperator{\sum_{\substack{k_j = 1\\j = 2, \dots, d}}^{N-1}} \Big( \sum_{j=2}^{d} k_j^2 \Big)^{-1}\\
	&\lesssim  \frac{2^d}{N^{d-2}} \smashoperator{\sum_{{\substack{k_l = 1\\l=2, \dots, d}}}^{ N-1}} \Big(\sum_{j=2}^{d} k_j^2\Big)^{-1/2}.
		% \leq {} & \frac{2^d}{N^{d-2}} \sum_{\substack{k_1 = 1}}^{ \lfloor (N-1)/2 \rfloor} k_1^2 \int_{x \in \R^{d-1}, \, |x| \leq \sqrt{d}N/2}\frac{1}{(k_1^2 + |x|^2)^2} \diff x \\
		% = {} & \frac{2^d}{N^{d-2}} \sum_{\substack{k_1 = 1}}^{ \lfloor (N-1)/2 \rfloor} k_1^{d-3} \int_{x \in \R^{d-1}, \, |y| \leq \sqrt{d}N/(2k_1)}\frac{1}{(1 + |y|^2)^2} \diff x \\
		% = {} & \frac{2^d}{N^{d-2}} \operatorname{vol}(S^{d-2}) \sum_{\substack{k_1 = 1}}^{ \lfloor (N-1)/2 \rfloor} k_1^{d-3} \int_{0}^{\sqrt{d}N(2k_1)} \frac{1}{(1 + r^2)^2} r^{d-2} \diff r \\
		% \leq {} & \frac{2^d}{N^{d-2}} \operatorname{vol}(S^{d-2}) \sum_{\substack{k_1 = 1}}^{ \lfloor (N-1)/2 \rfloor} k_1^{d-3} \int_{0}^{\sqrt{d}N(2k_1)} \frac{1}{r^{d-6}} \diff r \\
		% \lesssim {} & \frac{2^d}{N^{d-2}} \operatorname{vol}(S^{d-2}) \sum_{\substack{k_1 = 1}}^{ \lfloor (N-1)/2 \rfloor} k_1^{d-3} \left(\frac{\sqrt{d}N}{2k_1}\right)^{d-5}\\
		% \lesssim {} & \frac{2^d}{N^{d-2}} \operatorname{vol}(S^{d-2}) \sum_{\substack{k_1 = 1}}^{ \lfloor (N-1)/2 \rfloor} k_1^{d-3}  \\
		% = {} & 2^{d-1} \frac{N^2}{\pi} \sum_{\substack{k_j = 1\\ j = 1,\dots,d-1}}^{ \lfloor (N-1)/2 \rfloor} \left( \sum_{j=2}^{d} k_j^2 \right)^{-1/2}.
	\end{align}
% \ndpr{did you pay attention to constants in seris $\le $ integral?}
Next, bounded the series by an integral together with a change to polar coordinates, we get
		\begin{align}
\frac{2^d}{N^{d-2}} \sum_{\substack{k_l = 1\\ l = 2,\dots,d}}^{ N-1} \Big( \sum_{j=2}^{d} k_j^2 \Big)^{-1/2}
&\le  \frac{2^d}{N^{d-2}}\int_{\left\{\substack{0 \leq x_j \leq N, \; j = 1,\dots,d-1}\right\}} \frac{1}{\|x\|_2} \diff x\\
&\le \frac{2^d}{N}\int_0^N \int_0^N\frac{1}{\sqrt{x^2+y^2}} \diff x \diff y\\
%&=2^{d+1}\int_0^{\pi/4}\frac{1}{\cos(\theta)} \diff \theta 
&=2^{d}\log (3 +2\sqrt{2}) \le 2^d. % \qedhere
%&\le \frac{1}{N^{d-2}} \int_{\{x \in \R^{d-1}, \, \|x\|_2 \leq \sqrt{d} N\}} \frac{1}{\|x\|_2} \diff x\\
%&= \frac{1}{N^{d-2}} \operatorname{vol}(S^{d-2}) \int_{0}^{\sqrt{d}N} r^{d-3} \diff r \\
%&= \frac{\pi^{(d-2)/2}}{N^{d-2}(d-2)\Gamma((d-1)/2)} \big(\sqrt{d}N \big)^{d-2}\\
%&=\frac{\pi^{(d-4)/2}}{(d-2)\sqrt{2 \pi}((d-1)/2)^{d/2} \e^{-(d-1)/2}} d^{(d-2)/2}\le  8^d N^d,
\end{align}
%	The above sum can be bounded by an integral over the non-negative orthant, and in turn, by symmetry, we can write this as an integral over the \( \ell_\infty \)-ball of radius \( N \).
%	Bounding \( \| x \|_2 \leq \sqrt{d} \| x \|_\infty \) then allows us to rewrite it as a spherical integral.
%	\ndjc{get an additional \( N \) in the integral, but this gets canceled out by the fact that the sums you take out only run over \( d-3 \) dimensions}
%
%	\begin{align}
%\frac{2^d}{N^{d-2}} \sum_{\substack{k_l = 1\\ l = 2,\dots,d}}^{ N-1} \Big( \sum_{j=2}^{d} k_j^2 \Big)^{-1/2}
%&\le  \frac{2^d}{N^{d-2}}\int_{\left\{\substack{0 \leq x_j \leq N, \; j = 1,\dots,d-1}\right\}} \frac{1}{\|x\|_2} \diff x\\
%&\le \frac{1}{N^{d-2}} \int_{\{x \in \R^{d-1}, \, \|x\|_2 \leq \sqrt{d} N\}} \frac{1}{\|x\|_2} \diff x\\
%&= \frac{1}{N^{d-2}} \operatorname{vol}(S^{d-2}) \int_{0}^{\sqrt{d}N} r^{d-3} \diff r \\
%&= \frac{\pi^{(d-2)/2}}{N^{d-2}(d-2)\Gamma((d-1)/2)} \big(\sqrt{d}N \big)^{d-2}\\
%&=\frac{\pi^{(d-4)/2}}{(d-2)\sqrt{2 \pi}((d-1)/2)^{d/2} \e^{-(d-1)/2}} d^{(d-2)/2}\le  8^d N^d,
%	\end{align}
%
%	where we simplified the expression using the estimates \( \Gamma(k) \geq \sqrt{2 \pi} (k-1)^{k-1/2} \e^{-k+1} \) and \( \pi \in [3,4] \), \( \e \in [2,3] \).
\end{proof} 

%,
%\begin{align}
%	\label{eq:t}
%	\leadeq{\frac{2}{\sqrt{n}} \| \hat\theta - \bar{\theta} \| \frac{1}{\sqrt{n}} \left( \sigma + \sigma \sqrt{2 \log (1/\delta)} + 2 \lambda \kappa^{-1} \sqrt{|T|} \right)}\\
%	\leq {} & \frac{1}{n} \| \hat\theta - \bar{\theta} \|^2 + \frac{4 \sigma^2 (1 + 2 \log(1/\delta))}{n} + \frac{4 \sigma^2 \rho^2 \kappa^{-2} |T| \log(m/\delta)}{n},
%\end{align}
%the desired result.
%\end{proof}

\subsection{Control of the inverse scaling factor for \( C^k_n \)}
\label{sec:proof-power-graph}

In this subsection, we prove Proposition \ref{prp:power-graph} that we recall here for convenience.

\begin{proposition*}
	For \( G = C_n^k \) where \( k \leq n/2 \),  \( \rho \lesssim \sqrt{n}/k^3 + 1 \) and \( \kappa \gtrsim 1/\sqrt{k} \).
\end{proposition*}

\begin{proof}
	The bound on \( \kappa \) follows from Lemma \ref{lem:a} and the fact that the degree of \( C_n^k \) is bounded by \( 2k \).

	To bound \( \rho \), write \( D^\dag = [s_1, \dots, s_m] = (D^\top D)^\dag D^\top \) and \( D^\top = [d_1, \dots, d_m] \) and use the same technique and notation as in the proof of Proposition \ref{prp:tv-cols-2d} in Subsection \ref{proof:2d}.
	The Laplacian of \( C_n^k \) has the form of a circulant matrix whose first row is
	\begin{equation}
		\label{eq:da}
		% \rule[\dimexpr-4ex-\ht\strutbox]{0pt}{\dimexpr4ex+4ex+\baselineskip}
		a = \begin{bmatrix} 2k & \undermat{k \text{ times}}{-1 & \dots & -1} & 0 & \dots & 0 & \undermat{k \text{ times}}{-1 & \dots & -1} \end{bmatrix}\,.
		\rule[\dimexpr-4ex]{0pt}{4ex + \baselineskip}
	\end{equation}
	Hence, we can choose the discrete Fourier basis \( (v_m)_j = \exp(2 \pi \ic m j/n)\), \(m,j \in \llbracket 0, n \llbracket \) as an eigenbasis.
	The eigenvalues are given by
	\begin{equation}
		\label{eq:cw}
		\lambda_m = \sum_{l = 0}^{n - 1} e^{2 \pi \ic m l/n} a_l
		% =  2 k - \sum_{l = 1}^{k} \left( e^{2 \pi \ic l m/n} + e^{-2 \pi \ic l m/n} \right)
		=  2 \sum_{l = 1}^{k} \left( 1 - \cos \left( \frac{2 \pi l m}{n} \right) \right)\,.
	\end{equation}
	By the formula for the sums of squares,
	\begin{align}
		\label{eq:cx}
		\sum_{l = 1}^{k} l^2 = \frac{1}{6} k (k+1) (2k + 1) \geq \frac{1}{3} k^3\,,
	\end{align}
	and using the same estimates for the cosine as in Subsection \ref{proof:2d}, \( 2 - 2 \cos x \geq x^2/2 \) for \( x \in [0, 1/2] \), and \( 2 - 2 \cos x \geq 0.1 \) for \( x \in [1/2, \pi] \), we see that for \( 2 \pi l m/n \geq 1/2 \),
	\begin{equation}
		\label{eq:cy}
		2 \sum_{l = 1}^{k} \left( 1 - \cos \left( \frac{2 \pi l m}{n} \right) \right)
		\geq \frac{1}{2} \sum_{l=1}^{k} \left( \frac{2 \pi l m}{n} \right)^2 \geq \frac{k^3}{6} \left( \frac{2 \pi m}{n} \right)^2\,.
	\end{equation}
	Moreover, by the Lipschitz continuity of the exponential,
	\begin{equation}
		\label{eq:cz}
		| \langle v_m , d_j \rangle|^2 = \frac{1}{n} \left| e^{2 \pi \ic m (j + 1)/n} - e^{2 \pi \ic m j/n}\right|^2
		\le \frac{4 m^2 \pi^2}{n^3}\,. %\sin \frac{l \pi}{N} \xi
	\end{equation}
	By expressing the norm of the columns of \( D^\dag \) in terms of the eigendecomposition and combining pairs eigenvalues with the same value which have the same eigenvectors up to a sign in the exponential, we finally get
	\begin{align}
		\label{eq:cv}
		\| s_{j} \|_2^2
		= {} & \sum_{m = 1}^{n-1} \frac{1}{\lambda_m^2} \langle v_m , d_j \rangle^2\\
		% = \sum_{m = 1}^{n-1} \frac{1}{\lambda_m^2} \langle v_m , d_j \rangle^2\\
		\leq {} & 8 \frac{\pi^2}{n^3} \sum_{m = 1}^{\lceil (n-1)/2 \rceil} m^2 \left(2 \sum_{l = 1}^{k} 
			\left( 1 - \cos \left( \frac{2 \pi l m}{n} \right) \right)\right)^{-2}\\
		% \left( \1(2 \pi k m/n \leq 1/2) + \1(2 \pi k m/n > 1/2) \right)\\
		\lesssim {} & n \sum_{m=1}^{ \lceil n-1/(8 \pi k) \rceil} \frac{1}{m^2 k^6} + \frac{1}{n^3} \sum_{m=1}^{n} m^2
		\lesssim \frac{n}{k^3} + 1 \lesssim \frac{n}{k^6} + 1\,.% \qedhere
	\end{align}
\end{proof}

\subsection{Estimation rate for H\"older functions}
\label{sec:proof-holder-fcns}

In this subsection, we prove Proposition \ref{prp:holder-fcns} that we recall here for convenience.

\begin{proposition*}
	Fix \( \delta \in (0,1) \), \( d \geq 2 \), \( L > 0 \), $N \ge 1$, $n=N^d$ and \( \alpha \in (0,1] \) and let \( y \) be sampled according to the Gaussian sequence model \eqref{eq:al}, where \(\theta^*_{ \bm{i}} =  f^*(x_{\bm{i}}) \), \( \bm{i} \in [N]^d \) for some unknown function $f^*:[0,1]^d \to \R$.
	There exist positive constants $c$, $C$ and \( C' = C'(\sigma, L, d) \) such that the following holds.
Let $\hat \theta$ be the TV denoiser defined in~\eqref{eq:y} for the \( N^d \) grid with incidence matrix \( D_d \) and tuning parameter \( \lambda = c \sigma \sqrt{r_d(n) \log(e n/\delta)}/n, c>0 \) where $r_2(n)=\log n$ and $r_d (n)=1$ for $d \ge 3$. Moreover, let $\hat f:[0,1]^d \to \R$ be defined by $\hat f(x_{\bm{i}})=\hat \theta_{\bm{i}}$ for $\bm{i} \in [N]^d$ and arbitrarily elsewhere on the unit hypercube $[0,1]^d$.

Further, assume that \( N \geq C'(L, \sigma, d) \sqrt{r_d(n) \log(en/\delta)} \).
Then,
\begin{align}
%		 \| \widehat{f} - f^\ast \|_n^2
%		\leq \inf_{ \bar{f} \in H(\alpha, L)} \left\{ \| \bar{f} - f^\ast \|^2_n \right\} + C L^2 \left( \frac{\sigma^2 \sqrt{r_d(n) \log(e n/\delta)}}{L^2 n^{2/d}} \right)^{\frac{2 \alpha}{2 \alpha + 2}} + C \frac{\sigma^2 r_d(n) \log(e n/\delta)}{n^{2/d}},
		 \| \widehat{f} - f^\ast \|_n^2
		\leq {} & \inf_{ \bar{f} \in H(\alpha, L)} \left\{ \| \bar{f} - f^\ast \|^2_n \right\}
		+ C \frac{\big(L^2(\sigma\sqrt{r_d(n) \log(e n/\delta)})^{2\alpha} \big)^{\frac{1}{\alpha+1}}}{n^{\frac{2 \alpha}{d \alpha + d}}} + C \frac{\sigma^2}{n}\log (e/ \delta)\,,
	\end{align}
	with probability at least \( 1 - 2 \delta \)\,. %In particular, for $d=2$, it yields the near optimal rate
	% \begin{align}
	% 	\| \widehat{f} - f^\ast \|_n^2
	% 	\leq {} & \inf_{ \bar{f} \in H(\alpha, L)} \left\{ \| \bar{f} - f^\ast \|^2_n \right\}
	% 	+ C \big(L^2(\sigma \log(e n/\delta))^{2\alpha} \big)^{\frac{1}{\alpha+1}}n^{-\frac{2 \alpha}{2 \alpha + 2}}
	% 	+ C \frac{\sigma^2}{n}\log (e/ \delta)\,.
	% \end{align}
	\end{proposition*}

\begin{proof} Throughout this proof, it will be convenient to identify a function $g$ to the vector $(g(x_{\bm{i}}), {\bm{i}} \in [N]^d)$. 
	We use \eqref{eq:z} to get that for any vector $\bar f \in \R^{N^d}$, it holds
	\begin{equation}
		\label{eq:ad}
		\frac{1}{n}\| \widehat{f} - f^\ast \|^2 \leq  \frac1n\| \bar{f} - f^\ast \|^2_2 + 4 \lambda \| D \bar{f} \|_1 + C \frac{\sigma^2}{n}\log (e/ \delta)\,.
	\end{equation}
Denote by $\Theta(\alpha, L)$ the set of vectors on the grid that satisfy~\eqref{eq:au} and observe that it is a closed convex set so that \( f_{\mathrm{proj}} = \argmin_{\theta \in \Theta(\alpha, L)} \| \theta - f^\ast \|^2 \) is uniquely defined. Moreover, 	\begin{equation}
		\label{eq:az}
		\| \bar{f} - f^\ast \|^2 \leq \| f_{\mathrm{proj}} - \bar{f} \|^2 + \| f_{\mathrm{proj}} - f^\ast \|^2,
	\end{equation}
	which plugged back into \eqref{eq:ad} yields
	\begin{align}
		\label{eq:ba}
		\frac1n	 \| \widehat{f} - f^\ast \|^2
		\leq {} & \frac1n\inf_{f \in H(\alpha, L)} \| f - f^\ast \|^2
		+ \frac1n \| \bar{f} - f_{\mathrm{proj}}\|^2 + 4 \lambda \| D \bar{f} \|_1 + C \frac{\sigma^2}{n}\log (e/ \delta)\,.
	\end{align}
The remainder of the proof consists in choosing $\bar f$ to balance the approximation error and the stochastic error.

Fix an integer $k$ to be determined later and for any ${\bm i} \in [N]^d$, define $a_{\bm i}=k\lfloor\bm{i}/k\rfloor$. %	Divide the hypercube $[0,1]^d$ into \( k^d \) boxes
%	\begin{equation}
%		\label{eq:cc}
%		\cI_{ \bm{i}} = \llbracket a_{i_1}, a_{i_1 + 1} \llbracket \times \llbracket a_{i_2}, a_{i_2 + 1} \llbracket \times \dots \times \llbracket a_{i_d}, a_{i_d + 1} \llbracket, \quad \bm{i} \in [k]^d.
%	\end{equation}
%		with corner points \( a_{i_j} \) such that \( 0 \leq a_{i_j + 1} - a_{i_j} \leq 2 N/k \) for all \( j \in [d] \), \( i_j \in [k] \).
Next, define a piecewise constant approximation \( \bar{f} \) to \( f_{\mathrm{proj}} \) by $\bar{f}_{ \bm{i}} = (f_{\mathrm{proj}})_{a_{\bm{i}}}$ for $\bm{i}\in [N]^d$.

We first control the approximation error for all ${\bm i} \in [N]^d$ as follows:
%
%the diameter of each box is at most \( \frac{2 d N}{k} \), hence for any $x \in \cI_{ \bm{i}}$, it holds
\begin{align*}
| f_{\mathrm{proj}}( \bm{i}) - \bar f(\bm{i}) |&=| f_{\mathrm{proj}}( \bm{i}) - f_{\mathrm{proj}}(a_{\bm{i}})| \le LN^{-\alpha}\| \bm{i}-a_{\bm{i}}\|_\infty^\alpha \le L(k/N)^{\alpha}\,.
\end{align*}
%
%| f_{\mathrm{proj}}( \bm{i}) - f_{\mathrm{proj}}({a_{ \bm{i}}}) |
%		\leq {}  L^2 \left(\frac{2 d N}{k}\right)^{2\alpha} N^{-2 \alpha}
%		\leq (2d)^2 L^2 k^{-2 \alpha }\,,
%$$
It yields
$$
\frac{1}{n}\| \bar{f} - f_{\mathrm{proj}} \|_2^2\le L^2(k/N)^{2\alpha} \,.
%		(2d)^2 \frac{N^d}{n} L^2 k^{-2 \alpha}
%		= (2d)^2 L^2 k^{-2 \alpha}\,.
$$

Next, we control the term \( \| D \bar{f} \|_1 \). To that end, observe that if $ \bm{i}$ and $ \bm{i'}$ are neighbors in the grid, then 
$$
|\bar f_{ \bm{i}}- \bar f_{ \bm{i'}}|\le LN^{-\alpha}\1(a_{ \bm{i}}\neq a_{ \bm{i'}}).
$$
Therefore
$$
\| D \bar{f} \|_1 \le L(k/N)^{\alpha}\sum_{\bm{i}\sim \bm{i'}}\1(a_{ \bm{i}}\neq a_{ \bm{i'}}) \le 
L(k/N)^{\alpha}2dk^{d-1}\Big(\frac{N}{k}\Big)^d=2dL\frac{N^{d-\alpha}}{k^{1-\alpha}}\,.
$$
Hence
$$
	\lambda \| D \bar{f} \|_1 \lesssim \frac{L\sigma }{k^{1-\alpha}N^{\alpha}}\sqrt{r_d(n)\log(en/\delta)}.
		%\lesssim \frac{L^2 N^{-2\alpha}}{k^2} + \sigma^2 \frac{k^2}{N^{2}} r_d(n) \log (en/\delta),
$$
Choosing now 
$$
M = \left( \frac{\sigma N^\alpha \sqrt{r_d(n) \log (en/\delta)} }{L} \right)^{\frac{1}{ \alpha +1}}, \quad   k=\left\lceil M \right\rceil\,,
$$
yields the desired result, taking into account that \( M \in [1,N] \) if we assume
\begin{equation*}
	\label{eq:dc}
	N \geq \left( \frac{L}{\sigma \sqrt{r_d(n) \log(en/\delta)}} \right)^{1/\alpha} \vee \frac{\sigma \sqrt{r_d(n) \log(en/\delta)}}{L}\,.% \qedhere
\end{equation*}
\end{proof}

\subsection{Estimation rate for piecewise constant functions}
\label{sec:proof-est-pw-cst}

In this subsection, we prove Proposition \ref{prp:est-pw-cst} that we recall here for convenience.

\begin{proposition*}
		Fix \( \delta \in (0,1) \), \( d \geq 2 \),  $N \ge 1$, $n=N^d$  and let \( y \) be sampled according to the Gaussian sequence model \eqref{eq:al}, where \(\theta^*_{ \bm{i}} =  f^*(x_{\bm{i}}) \), \( \bm{i} \in [N]^d \) for some unknown function $f^*\in PC(d, \beta)$.
		There exist positive constants $c$ and $C$ such that the following holds.
Let $\hat \theta$ be the TV denoiser defined in~\eqref{eq:y} for the $d$-dimensional grid with incidence matrix \( D_d \) and tuning parameter \( \lambda = c \sigma \sqrt{r_d(n) \log(e n/\delta)}/n \), where $r_2(n)=\log n$ and $r_d (n)=1$ for $d \ge 3$.
Moreover, let $\hat f:[0,1]^d \to \R$ be defined by $\hat f(x_{\bm{i}})=\hat \theta_{\bm{i}}$ for $\bm{i} \in [N]^d$ and arbitrarily elsewhere on the unit hypercube $[0,1]^d$.
Then,
%
%
%	Let \( \beta > 0 \) and \( f^\ast \in PC(\beta) \) be piecewise constant.
%	Then, there are constants \( C = C(d), c = c(d) > 0 \) such that for any tolerance level \( \delta \in (0,1) \), with probability at least \( 1 - 2 \delta \), the TV denoiser \( \widehat{f} \) on a regular grid with \( \lambda = c \sigma \sqrt{r_d(n) \log (e n/\delta)}\) achieves
	\begin{equation}
		\| \widehat{f} - f^\ast \|_n^2 \lesssim \frac{\sigma^2 \beta}{n^{1/d}} r_d(n) \log(e n/\delta) +  \frac{\sigma^2}{n}\log(e /\delta)\,,
	\end{equation}
	with probability at least \( 1 - 2 \delta \).
\end{proposition*}

\begin{proof}
For any $x \in \R^d$, and any closed set $B \subset \R^d$ define the distance from $x$ to $B$ by $d(x,B)=\min_{b \in B}\|x-b\|$. Next define
	\begin{equation}
		\label{eq:bw}
		T := \{ ( \bm{i}, \bm{j}) : \bm{i}\sim\bm{j} \text{ and } \min \{ d(x_{ \bm{i}}, B(f)), d(x_{ \bm{j}}, B(f)) \} \leq 4/N \}
	\end{equation}
	to be the set of edges whose nodes are close to the boundary $B(f)$.
It can be readily checked that the vector \( \big(f(x_{ \bm{i}}), \bm{i} \in V\big)  \) is constant on the connected components of \( (V, E \setminus T) \).

First, let us state a lemma that allows us to bound the number of grid points in a neighborhood of a set by the volume of said set. 

\begin{lemma}
	\label{lem:b}
	\cite[Lemma 8.3]{arias-castro_oracle_2012}
	Let \( B \subseteq [0,1]^d \), \( A = [0,1]^d \cap \big(B+\cB(\eta)\big) \), \( 4/N \leq \eta \leq 1 \). 
	Then, the number of grid points on a regular $d$-dimensional grid \( \mathcal{X}_N^d \) intersecting \( A \) is bounded by
	\begin{equation}
		\label{eq:bs}
		8^{-d}N^d \operatorname{vol}(A) \leq | A \cap \mathcal{X}_N^d | \leq 4^d N^d \operatorname{vol}(A).
	\end{equation}
\end{lemma}
%
%	Recall the notation for the \( \eta \) neighborhood of a set, \( B(A, \eta) = \{ x \in [0,1]^d : d(x, A) \leq \eta \} \).
	By Lemma \ref{lem:b}, Definition \ref{def:box-counting-dim} and the triangle inequality, we get
\begin{align}
		|T| &\le  \big|\mathcal{X}_N^d \cap \big(B(f)+\cB(4/N)\big)\big| \leq 4^d N^d \operatorname{vol}(B(f)+\cB(4/N))\\
	&\le 4^d N^d \beta \operatorname{vol}(B(1)) (8/N)^d (8/N)^{-(d-1)} \leq C(d) \beta N^{d-1} \label{eq:bu}\,,
	\end{align}
where $C(d)$ is a dimension-dependent constant.

	Since \( f^\ast \) is constant along all edges not included in \( T \), \( \| (D f^\ast)_{T^c} \|_1 = 0 \).
 	Taking into account \( n = N^d \), Corollaries~\ref{cor:oracle-tv-2d} and~\ref{cor:oracle-tv} readily yield the desired result.
%	\begin{equation}
%		\label{eq:bt}
%		 \| \widehat{f} - f^\ast \|_n^2 \leq C(d) \frac{\sigma^2 \beta}{n^{1/d}}r_d(n) \log(en/\delta) + C(d) \frac{\sigma^2}{n}.
%	\end{equation}
%	with probability at least \( 1 - 2 \delta \).
\end{proof}

\subsection{Estimation rate for cartoon functions}
\label{sec:proof-cartoon-fcns}

In this subsection, we prove Proposition \ref{prp:cartoon-fcns} that we recall here for convenience.

\begin{proposition*}
		Fix \( \delta \in (0,1) \), \( d \geq 2 \),  $N \ge 1$, $n=N^d$  and let \( y \) be sampled according to the Gaussian sequence model \eqref{eq:al}, where \(\theta^*_{ \bm{i}} =  f^*(x_{\bm{i}}) \), \( \bm{i} \in [N]^d \) for some unknown function $f^*\in PH(d,\beta, \alpha, L)$,  $\alpha \in (0,1]$, $L>0$, $\beta>0$.
		There exist positive constants $c$, $C$ and \( C' = C'(\sigma, L, d) \) such that the following holds.
		Let $\hat \theta$ be the TV denoiser defined in~\eqref{eq:y} for the $d$-dimensional grid with incidence matrix \( D_d \) and tuning parameter \( \lambda = c \sigma \sqrt{r_d(n) \log(e n/\delta)}/n \), where $r_2(n)=\log n$ and $r_d (n)=1$ for $d \ge 3$.
		Moreover, let $\hat f:[0,1]^d \to \R$ be defined by $\hat f(x_{\bm{i}})=\hat \theta_{\bm{i}}$ for $\bm{i} \in [N]^d$ and arbitrarily elsewhere on the unit hypercube $[0,1]^d$.

If \( N \geq C'(L, \sigma, d) \sqrt{r_d(n) \log(en/\delta)} \), then
%
%	Let \( d \geq 2 \), \( \alpha \in [0,1] \), \( \beta, L > 0 \), \( f^\ast \in PH(d,\beta, \alpha, L)\) be piecewise H\"older.
%	Then, there are constants \( C, c > 0 \) depending on \( d \) such that the fused Lasso estimator \eqref{eq:y} for the \( N^d \) grid \( \mathcal{X}_N^d \) with incidence matrix \( D_d \) and tuning parameter \( \lambda = \sigma c \sqrt{r_d(n) \log(n/\delta)} \) for any tolerance level \( \delta \in (0,1) \) has rate
	% \label{prp:holder-rate}
	\begin{align}
		\frac{1}{n} \| \widehat{f} - f^\ast \|^2
		\lesssim {} & \frac{\big(L^2(\sigma\sqrt{r_d(n) \log(e n/\delta)})^{2\alpha} \big)^{\frac{1}{\alpha+1}}}{n^{\frac{2 \alpha}{d \alpha + d}}}
		+  \frac{\sigma^2\beta}{n^{1/d}}r_d(n) \log(en\delta)+\frac{\sigma^2}{n}\log(e /\delta)\,,
	\end{align}
	with probability at least \( 1 - 2 \delta \).
\end{proposition*}

\begin{proof}
	As in the proof of Proposition \ref{prp:est-pw-cst} (Subsection \ref{sec:proof-est-pw-cst}), in Corollaries~\ref{cor:oracle-tv-2d} and~\ref{cor:oracle-tv}, set
	\begin{equation}
		\label{eq:bx}
		T := \{ ( \bm{i}, \bm{j}) : \bm{i} \text{ neighbor of } \bm{j} \text{ and }  d(x_{ \bm{i}}, B(f))\wedge d(x_{ \bm{j}}, B(f))  \leq 4/N \},
	\end{equation}
and note that $|T| \leq C(d) \beta N^{d-1}$, using the same argument as in~\eqref{eq:bu}. Moreover, it can be readily checked that the vector \( \big(f(x_{ \bm{i}}), \bm{i} \in V\big)  \) satisfies the H\"older condition~\eqref{eq:au} on the connected components of \( (V, E \setminus T) \). 

Next, we adopt the same discretization of as in Proposition~\ref{prp:holder-fcns}, with a slight modification to take into account that \( f^\ast \) is only Hölder-continuous within connected components of the underlying grid.
To that end, fix an integer $k$ to be determined later and for any ${\bm i} \in [N]^d$, define indices \( a_{ \bm{i}} \) and corresponding boxes \( A_{ \bm{i}} \) by
\begin{equation}
	\label{eq:ct}
	(a_i)_j = \left\{
	\begin{aligned}
		k \lfloor i_j / k \rfloor, \quad &i_j \leq N,\\
		N, \quad &i_j = N + 1,
	\end{aligned}
	\right. \quad
	A_{ \bm{i}} = \llbracket a_{i_1}, a_{i_1 + 1} \llbracket \times \dots \times \llbracket a_{i_d}, a_{i_d + 1} \llbracket\,.
\end{equation}
For each of the boxes \( A \) and every connected component \( C \) of \( (V, E \setminus T) \) within, pick a fixed representative \( b(C) \) and write \( C( \bm{i}) \) for the connected component in \( A_{ \bm{i}} \) that \( \bm{i} \) belongs to.
%	Divide the hypercube $[0,1]^d$ into \( k^d \) boxes
%	\begin{equation}
%		\label{eq:cc}
%		\cI_{ \bm{i}} = \llbracket a_{i_1}, a_{i_1 + 1} \llbracket \times \llbracket a_{i_2}, a_{i_2 + 1} \llbracket \times \dots \times \llbracket a_{i_d}, a_{i_d + 1} \llbracket, \quad \bm{i} \in [k]^d.
%	\end{equation}
%		with corner points \( a_{i_j} \) such that \( 0 \leq a_{i_j + 1} - a_{i_j} \leq 2 N/k \) for all \( j \in [d] \), \( i_j \in [k] \).
Next, define a piecewise constant approximation \( \bar{f} \) to \( f^* \) by $\bar{f}_{\bm{i}} = f^*_{ b(C(\bm{i}))}$ for $\bm{i}\in [N]^d$. 

Using the same arguments as in the proof of Proposition~\ref{prp:holder-fcns} (Subsection \ref{sec:proof-holder-fcns}), we get first that $n^{-1}\| \bar{f} - f^* \|_2^2\le L^2(k/N)^{2\alpha}$ and second that
$$
\| (D \bar{f})_{T^c} \|_1 \le 2dL\frac{N^{d}}{k^{1-\alpha}N^\alpha}\,.
$$
Choosing now 
$$
k=\left\lceil\Big( \frac{\sigma N^\alpha \sqrt{r_d(n) \log (en/\delta)} }{L} \Big)^{\frac{1}{ \alpha +1}} \right\rceil
$$
and applying  Corollaries~\ref{cor:oracle-tv-2d} and~\ref{cor:oracle-tv} yields the desired result.
\end{proof}

\section{Rates for Haar wavelet thresholding}
\label{sec:haar-thresholding}

Interestingly, despite inferior performance in practice \cite{NeeWar13a}  Haar wavelet thresholding in dimension \( d \geq 2 \) yields similar rates to the ones we obtained in Corollaries \ref{cor:oracle-tv-2d} and \ref{cor:oracle-tv}, which we will show here in the 2D case.
It is a consequence of \cite[Proposition 7]{NeeWar13a}.

First, let us recall the notation from \cite{NeeWar13a} for the Haar basis.
In one dimension, the Haar wavelets are defined by considering the constant function \( H^0 \) on \( [0,1] \), 
\begin{equation}
	\label{eq:bg}
	H^0(t) = \left\{
	\begin{aligned}
		1, \quad & 0 \leq t < 1,\\
		0, \quad & \text{otherwise},
	\end{aligned}
	\right.
\end{equation}
and the mother wavelet \( H^1 \),
\begin{equation}
	\label{eq:bl}
	H^1(t) = \left\{
	\begin{aligned}
		1, \quad & 0 \leq t < 1/2,\\
		-1, \quad & 1/2 \leq t < 1,
	\end{aligned}
	\right.
\end{equation}
which is dilated and translated to get
\begin{equation}
	\label{eq:bm}
	H_{m, k}(t) = 2^{m/2} H^1(2^m t - k), \quad m \in \mathbb{N},\, 0 \leq k < 2^m\,.
\end{equation}
This collection of functions is an orthonormal basis of \( L_2([0,1)) \).
The bivariate Haar basis is then obtained by tensorization, setting
\begin{equation}
	\label{eq:dg}
	H^e(u,v) = H^{e_1}(u)H^{e_2}(v), \quad e = (e_1, e_2) \in V := \left\{ \{0, 1\}, \{1, 0\}, \{1, 1\} \right\}\,,
\end{equation}
and
\begin{equation}
	\label{eq:dh}
	H_{j, k}^e(x) = 2^{j} H^e(2^j x - k), \quad e \in V,\, j \geq 0,\, k \in \mathbb{Z}^2 \cap 2^j Q\,,
\end{equation}
where \( Q = (0, 1]^2 \), which again form an orthonormal basis of \( L_2(Q) \).
From there, by sampling on the grid \( \mathcal{X}^2_N \), \( N = 2^m \) we get discrete signals
\begin{align}
	(h^0)_{i_1, i_2} = {} & H^0((i_1 - 1)/N, (i_2 - 1)/N),\\
 	(h^e_{j, k})_{i_1, i_2} = {} & H^e_{j, k}((i_1 - 1)/N, (i_2 - 1)/N), \quad (i_1, i_2) \in \mathcal{X}_N^2.
\end{align}
such that
$$ \{ h^0 \} \cup \{ h^e_{j, k} \}^{e \in V}_{0 \leq j \leq n-1,\, k \in \mathbb{Z}^2 \cap 2^j Q}
$$ 
is an orthonormal basis of \( \R^{N^2} \).
Collecting the coefficients of these vectors into a matrix \( O \), we define the bivariate Haar wavelet transform by \( \mathcal{H}(y) = O^\top y \).

Second, we  use the performance of signal thresholding from \cite{DonJoh94} and the weak \( \ell_1 \) estimate for the Haar coefficients in terms of the TV norm from \cite{NeeWar13a}.

\begin{lemma}
	\label{lem:c}
	\cite[Theorem 1]{DonJoh94}
	Let \( y \) be drawn from the Gaussian sequence model \eqref{eq:al} and denote by \( \eta(y) \) the soft thresholding estimator defined by
	\begin{equation}
		\label{eq:w}
		\eta(y)_j = \operatorname{sgn}(y_j) ((|y_j| - \sigma \sqrt{2 \log n}) \vee 0)\,.
	\end{equation}
	Then,
	\begin{equation}
		\frac{1}{n} \E \| \eta(y) - \theta^\ast \|_2^2 \leq \frac{2 \log n + 1}{n} \left(\sigma^2 + \sum_{i = 1}^n ({\theta^\ast_i}^2 \wedge \sigma^2) \right)\,.
	\end{equation}
\end{lemma}

\begin{lemma}
	\label{lem:d}
	\cite[Proposition 7]{NeeWar13a}
	Write \( D_2 \) for the incidence matrix of the 2D grid, let \( \theta \) have zero mean and let \( c_{(k)} \) denote the \( k \)th largest entry of the bivariate Haar transform \( \mathcal{H}(\theta) \).
	Then, there exists a constant \( C > 0 \) such that
	\begin{equation}
		\label{eq:s}
		| c_{(k)} | \leq C \frac{\| D_2 \theta \|_1}{k}\,.
	\end{equation}
\end{lemma}

\begin{proposition}
	Let \( y \) be a sample of the Gaussian sequence model \eqref{eq:al} with \( n = N^2 \), denote by \( \widehat{\theta} = \mathcal{H}^{-1} \circ \eta \circ \mathcal{H}(y) \) the soft thresholder in the bivariate Haar wavelet basis and by \( D_2 \) the incidence matrix of the \( N \times N \) grid. Then,
	\begin{equation}
		\label{eq:aw}
		\E \MSE( \widehat{\theta} ) \lesssim \frac{\log n}{n} ( \sigma^2 + \sigma \| D_2 \theta^\ast \|_1 ),
	\end{equation}
	for \( n \) large enough.
\end{proposition}

\begin{proof}
  Since the Haar transform is orthogonal, we can apply Lemma \ref{lem:c} to the thresholding of the Haar coefficient vector of \( y \).
	Writing \( c^\ast_k \) for the coefficients \( c^\ast = \mathcal{H}(\theta^\ast) \), we have
	\begin{align}
		\label{eq:aa}
		\frac{1}{n} \E \| \widehat{\theta} - \theta^\ast \|_2^2
		= {} & \frac{1}{n} \E \| \eta( \mathcal{H}(y)) - \mathcal{H}(\theta) \|_2^2
		\lesssim \frac{\log n}{n} \left( \sigma^2 + \sum_{i = 0}^{n-1} ({c^\ast_i}^2 \wedge \sigma^2) \right)\\
		\leq {} & \frac{\log n}{n} \left( 2 \sigma^2 + \sum_{i = 1}^{n-1} ({c^\ast_i}^2 \wedge \sigma^2) \right)\,.
	\end{align}
	Now, since \( c^\ast_0 \) is to the mean of \( \theta^\ast \), the remaining Haar coefficients correspond to a mean zero vector, so we can write them in descending order as \( c^\ast_{(i)} \), apply Lemma \ref{lem:d} and introduce a cut-off at \( k \) to obtain
	\begin{align}
		\label{eq:an}
		\frac{1}{n} \E \| \widehat{\theta} - \theta^\ast \|_2^2
		\lesssim {} & \frac{\log n}{n} \left( 2 \sigma^2 + \sum_{i = 1}^{n-1} ({c^\ast_{(i)}}^2 \wedge \sigma^2) \right) \\
		\leq {} & \frac{\log n}{n} \left( 2 \sigma^2 + \sum_{i = 1}^{n-1} (\| D_2 \theta^\ast \|_1^2 \frac{1}{i^2} \wedge \sigma^2) \right) \\
		\leq {} & \frac{\log n}{n} \left( 2 \sigma^2 + \left(k \sigma^2 + \sum_{i = k+1}^{n-1} \frac{\| D_2 \theta^\ast \|^2_1}{i^2} \right) \right) \\
		\lesssim {} & \frac{\log n}{n} \left( 2 \sigma^2 + \left(k \sigma^2 + \| D_2 \theta^\ast \|_1^2 \int_{k+1}^\infty x^{-2} \diff x \right) \right) \\
		= {} & \frac{\log n}{n} \left( 2 \sigma^2 + \left(k \sigma^2 + \frac{\| D_2 \theta^\ast \|_1^2}{k + 1} \right) \right)\,. \\
	\end{align}
	Provided \( n \) is large enough, choosing \( k := \lfloor \| D_2 \theta^\ast \|_1 / \sigma \rfloor \) then yields
	\begin{align}
		\label{eq:aq}
		\frac{1}{n} \E \| \widehat{\theta} - \theta^\ast \|_2^2
		\lesssim {} & \frac{\log n}{n} (\sigma^2 + \sigma \| D_2 \theta^\ast \|_1).% \qedhere
	\end{align}
\end{proof}

Note that for the sake of a simple presentation, we phrased this result in terms of the (expected) mean squared error, but similar bounds hold with high probability and allowing misspecification, as well as an \( \ell_0 \)-\( \ell_1 \) trade-off.

Using Lemma \ref{lem:d}, we can also show a version of Corollary \ref{cor:oracle-tv-2d} that has an additional log factor.
\begin{proposition}
\label{prp:haar-tv-2d}
Let \( D_2 \) denote the incidence matrix of the 2D grid.
Then, there exist constants \( C, c > 0 \) such that the TV denoiser $\hat \theta$ defined in  \eqref{eq:y} with \( \lambda = c \sigma \log n \sqrt{\log (e n/\delta)}/n \) satisfies
	$$
	\E \mse(\hat \theta) \lesssim  \frac{ \sigma\|D_2\theta^*\|_1 \wedge \sigma^2\|D_2\theta^*\|_0 + 1}{n}\log^3 (en),
	$$
	where $\|D\theta^*\|_0$ denotes the number of nonzero components of $D\theta^*$.
\end{proposition}

\begin{proof}
	We follow the proof of Corollary \ref{cor:oracle-tv-2d} given in Section \ref{proof:main} and only indicate where we use Lemma \ref{lem:d} instead of controlling \( \rho \).
	Recall that \( \mathcal{H} \) is an orthogonal operator whose first coordinate corresponds to the mean of a vector.
	Starting from \eqref{eq:ar}, we get
	\begin{align}
		% \label{eq:di}
		\varepsilon^\top (\widehat{\theta} - \bar{\theta}) = {} & \mathcal{H}(\varepsilon)^\top (\mathcal{H}(\widehat{\theta} - \bar{\theta}))\\
		= {} & \sum_{i = 0}^{n-1} (\mathcal{H}(\varepsilon)_i) (\mathcal{H}(\widehat{\theta} - \bar{\theta})_i)\\
		\leq {} & | \mathcal{H}(\varepsilon)_0 | \| \widehat{\theta} - \bar{\theta} \|_2 + \max_{i = 1, \dots, n-1} | \mathcal{H}(\varepsilon)_i | \| D_2 (\widehat{\theta} - \bar{\theta}) \|_1 \sum_{i = 1}^{n - 1} \frac{1}{i}\\
		\leq {} & | \mathcal{H}(\varepsilon)_0 | \| \widehat{\theta} - \bar{\theta} \|_2 + \max_{i = 1, \dots, n-1} | \mathcal{H}(\varepsilon)_i | \| D_2 (\widehat{\theta} - \bar{\theta}) \|_1 \log n\,.
	\end{align}
	Now, use the same bounds for the 2-norm and the \( \infty \)-norm of independent Gaussians as after \eqref{eq:b}, noting that \( \mathcal{H}(\varepsilon) \) is again an isotropic Gaussian variable, so \( \mathcal{H}(\varepsilon)_0 \) is a one-dimensional projection of a Gaussian and \( \sup_{i = 1, \dots, n-1} | \mathcal{H}(\varepsilon)_i | \) the maximum of \( n - 1 \) independent Gaussians.
	The proof then continues as in Section \ref{proof:main}.
\end{proof}

\end{document}